\documentclass[11pt]{article}
\usepackage{sectsty}
\usepackage{enumitem}
\usepackage{titlesec}
\usepackage{mathtools}
\usepackage[margin=1in]{geometry}
\usepackage[all]{xy}
\usepackage{relsize}%
\usepackage{blindtext}
\usepackage{comment}
\usepackage[utf8]{inputenc}
\usepackage{amsmath,amsfonts,amssymb,amsthm,epsfig,epstopdf,titling,url,array}
\usepackage[hidelinks]{hyperref}
\usepackage[capitalize,noabbrev]{cleveref}

\crefformat{equation}{(#2#1#3)}
\crefrangeformat{equation}{(#3#1#4)–(#5#2#6)}
\crefmultiformat{equation}{(#2#1#3)}{, (#2#1#3)}{, (#2#1#3)}{, (#2#1#3)}
\crefrangemultiformat{equation}{(#3#1#4)–(#5#2#6)}{, (#3#1#4)–(#5#2#6)}{, (#3#1#4)–(#5#2#6)}{, (#3#1#4)–(#5#2#6)}

\numberwithin{equation}{section}
\newtheorem{thm}{Theorem}[section]
\newtheorem{defn}[thm]{Definition}
\newtheorem{cor}[thm]{Corollary}
\newtheorem{lem}[thm]{Lemma}
\newtheorem{claim}[thm]{Claim}
\begin{document}
	
	\title {Cohomogeneity One Expanding Ricci Solitons and the Expander Degree}
	\author{Abishek Rajan\thanks{This work was supported by the National Science Foundation under Grant No.\ DMS-2204364.}}
	
	\maketitle
	
	\begin{abstract}
		\noindent ABSTRACT. We consider the space of smooth gradient expanding Ricci soliton structures on $S^1 \times \mathbb{R}^3$ and $S^2 \times \mathbb{R}^2$ which are invariant under the action of $\text{SO}(3) \times \text{SO}(2)$. In the case of each topology, there exists a $2$-parameter family of cohomogeneity one solitons asymptotic to cones over the link $S^2 \times S^1$, as constructed in \cite{NW24}, \cite{Win21}, and \cite{BDGW15}. By analyzing the resultant soliton ODEs, we reconstruct the $2$-parameter families in each case and provide an alternate proof of conicality. Analogous to \cite{BC23}, we define a notion of \textit{expander degree} for these cohomogeneity one solitons through a properness result. We then proceed to calculate this cohomogeneity one expander degree in the cases of the specific topologies.
	\end{abstract}
	
	
	\tableofcontents
	
	\section{Introduction}
	
	Ricci flow, introduced by Hamilton in \cite{Ham82}, has had a significant impact on geometry and topology. A large body of applications are in dimension $3$, perhaps most notably the Poincar\'e conjecture, in which the singularity models of Ricci flow were classified by Perelman in \cite{Per02, Per03}. Perelman developed a surgery process to continue the flow beyond each of the possible singularities in $3$-dimensions. \\
	
	It is natural to try to extend these results to $4$ and higher dimensions in hope of constructing a Ricci flow through singularities in all dimensions. However, work of Bamler in \cite{Bam20a, Bam20b} show that the structures of singularities in $4$-dimensional Ricci flow are far more complicated. In particular, the singularity models may be conical. It is hoped that these conical singularities can be resolved by expanding Ricci solitons asymptotic at infinity to the given cones. \\
	
	Thus, the question of whether there exists an expanding Ricci soliton asymptotic to a given cone naturally arises. Work of Bamler and Chen in \cite{BC23} constructs a degree theory for asymptotically conical gradient expanding Ricci solitons. The central quantity, called the \textbf{expander degree} of a given compact $4$-orbifold with boundary, is essentially a signed count of the number of gradient expanding solitons defined on the interior of the orbifold which are asymptotic to any given cone with non-negative scalar curvature. This quantity is independent of the geometry of the chosen cone. Importantly, if the expander degree of a given orbifold is not $0$, then it is possible to construct gradient expanding Ricci solitons asymptotic to any cone with positive scalar curvature. \\
	
	In this paper, we define an analogous quantity which we call the \textbf{cohomogeneity one expander degree}, denoted $\textnormal{deg}^{\textnormal{sym}}_\textnormal{exp}$. Making this definition involves a certain properness result, whose proof takes up a bulk of this work. \\
	
	Our main results are the following:
	\begin{thm}\label{thm:mainthm1}
		$\textnormal{deg}^{\textnormal{sym}}_\textnormal{exp}(S^1 \times \mathbb{D}^3) = 1$
	\end{thm}
	
	\begin{thm}\label{thm:mainthm2}
		$\textnormal{deg}^{\textnormal{sym}}_\textnormal{exp}(S^2 \times \mathbb{D}^2) = 0$
	\end{thm}
	
	To prove this theorem, we construct a $2$-parameter family of  gradient expanding Ricci solitons each with an isometric action of $\text{SO}(3) \times \text{SO}(2)$ over the topologies $S^1 \times \mathbb{R}^3$ and $S^2 \times \mathbb{R}^2$. We note that the same solitons were originally constructed and analyzed in \cite{BDGW15}, \cite{Win21}, and \cite{NW24} using a different coordinate system. In the case of each topology, the metric can be written as a doubly warped product
	
	\begin{equation}\label{eq:metric}
		g = dr^2 + a(r)^2g_{S^1}+b(r)^2g_{S^2}
	\end{equation} 
	
	\noindent for smooth functions $a$ and $b$ with a soliton potential function $f$ which is also invariant under the group action. The high degree of symmetry possessed by these solitons implies that the expanding soliton equation $\text{Ric}_g + \nabla ^2 {f} + g = 0$ reduces to a system of $3$ ordinary differential equations in $a,b$ and $f$. It is possible to ensure that a soliton has the required topology by setting the initial conditions to the ODEs (at $r=0$) appropriately. In both cases, one of the initial conditions is the value of $f''(0)$; as the soliton equations are degenerate at $r=0$, $f''(0)$ must be specified in order to obtain a unique solution. \\
	
	To prove \Cref{thm:mainthm1} and \Cref{thm:mainthm2}, we first reconstruct the solitons from \cite{BDGW15}, \cite{Win21}, and \cite{NW24} using the coordinate system of \cite{A17}. Along the way, we prove the following theorem, which is a special case of the aforementioned work. We wish to point out that while this theorem is already known, the estimates we use in our alternate proof will be of further use when defining and calculating the expander degree.
	
	\begin{thm}\label{thm:nwtheorem}
		Suppose $M$ is diffeomorphic to either \( S^1 \times \mathbb{R}^3 \) or \( S^2 \times \mathbb{R}^2 \). There exists a two-parameter family of complete gradient expanding Ricci solitons on $M$, each invariant under the standard action of \(\mathrm{SO}(3) \times \mathrm{SO}(2)\). Further, these solitons are asymptotic to cones over the link \( S^2 \times S^1 \).
	\end{thm}

	In each case, the $2$ parameters are the initial conditions of the soliton equations. In the $S^1 \times \mathbb{R}^3$ case, the pair of initial conditions is $(a(0), f''(0))$, where $a(0) \equiv a_0$ is the size of the $S^1$ orbit at $r=0$ in (1.1) and $f''(0) \equiv f_0$ is as described above, while in the $S^2 \times \mathbb{R}^2$ case, the pair of initial conditions is $(b(0), f''(0))$ where $b(0) \equiv b_0$ is the size of the $S^2$ orbit at $r=0$. In both cases, as the constructed solitons are asymptotic to cones, the functions $a(r)$ and $b(r)$ are asymptotic to linear functions, whose slopes we denote $a'_\infty$ and $b'_\infty$, respectively. The corresponding cone metric is $\gamma = ds^2 + (a'_\infty s)^2 g_{S^1} + (b'_\infty s)^2 g_{S^2}$. \\

	We will show that for solitons of bounded curvature, the condition $f_0 < 0$ (along with either $a_0 > 0$ or $b_0 > 0$) is necessary and sufficient for a complete solution, in which case $a$ and $b$ are asymptotically linear. Note that as our goal is to analyze asymptotically conical solitons, we do not lose anything by assuming bounded curvature.\\
	
	Thus, in the case of either topology, we can consider the map $F: \mathbb{R}^+ \times \mathbb{R}^+ \to \mathbb{R}^+ \times \mathbb{R}^+$ which takes the initial conditions $(a_0, -f_0)$ (in the $S^1 \times \mathbb{R}^3$ case) or $(b_0, -f_0)$ (in the $S^2 \times \mathbb{R}^2$ case) to the slopes $(a'_\infty,b'_\infty)$. \\

	We further show that the asymptotic cone of the soliton varies continuously in the initial conditions, which amounts to showing that the map $F$ is continuous. Further, we show that $F$ is a proper map. This allows us to define the degree of the map $F$. \\
	
	In \cite{BC23}, an invariant called the \textbf{expander degree}, denoted $\text{deg}_\text{exp}$, was defined for a certain class of compact smooth $4$-orbifolds with boundary. The expander degree of such an orbifold roughly counts (with sign), for any fixed cone metric $\gamma$, the number of gradient expanding solitons defined on the interior of the orbifold with non-negative scalar curvature which are asymptotic to $\gamma$. \\
	
	Analogous to \cite{BC23}, in this paper, we define the \textbf{cohomogeneity one expander degree}, denoted $\text{deg}^{\text{sym}}_\text{exp}$, of the orbifolds $S^1 \times \mathbb{D}^3$ and $S^2 \times \mathbb{D}^2$, as the degree of the map $F$ in the $S^1 \times \mathbb{R}^3$ and the $S^2 \times \mathbb{R}^2$ cases, respectively. We note that we do not require the hypothesis of nonnegative scalar curvature as in \cite{BC23}, although we reiterate that our results are only valid for cohomogeneity one solitons. Similar to the general case, $\text{deg}^{\text{sym}}_\text{exp}$ represents the number (counted with sign) of cohomogeneity one gradient expanding solitions defined on the interior of the orbifold which are asymptotic to a fixed cone metric.   \\
	
	The proof of \Cref{thm:nwtheorem} is based on understanding the behaviors of the profile functions $a$ and $b$ in \Cref{eq:metric}; namely, that these functions are increasing. Additionally, the soliton potential $f$ is non-positive with non-positive first and second derivatives. Putting these inequalities together, we show that the functions $a,b$ and $f$ can be extended to the interval $[0,\infty)$, showing that the corresponding solitons are complete. We note that these solitons were constructed previously and analyzed as particular cases of a more general method in \cite{NW24, Win21, BDGW15}. However, the methods we use are more amenable to proving certain curvature estimates and proving a properness result used to define the cohomogeneity one expander degree.  \\
	
	Then, we understand how the curvature of the soliton decays at infinity. We show that $|\text{Rm}_g| \leqslant {C}/{r^2}$, where $C$ is a constant that depends continuously on the initial conditions for either topology. From this, we show that the slopes $a'(r)$ and $b'(r)$ respectively converge to finite positive limiting values $a'_\infty$ and $b'_\infty$ as $r \to \infty$, indicating that the soliton metric $g$ is asymptotic to the cone metric $\gamma = ds^2 + (a'_\infty s)^2 g_{S^1} + (b'_\infty s)^2 g_{S^2}$. We then extend this result to show that the asymptotic cone of an expanding soliton varies continuously as the soliton varies; this translates into the fact that the slopes $a'_\infty$ and $b'_\infty$ are continuous functions of the initial conditions. \\
	
	Having done the above, we verify that in each case, the respective map $F$ is proper using tools from \cite{BC23} and proceed to calculate the cohomogeneity one expander degree individually in both cases. \\
	
	In Sections 2,3 and 4, we derive the soliton equations and prove the monotonicity of the functions $a,b$ and $f$. In Section $5$, we show that the solitons in each case are complete and form a $2$-parameter family. In Section $6$ and $7$, we show that the solitons are asymptotically conical and that the slopes vary continuously in the parameters. Then, we define the map $F$ and show that it is continuous. Additionally, we quantify how close an asymptotic cone metric is to the corresponding expanding soliton in terms of the geometry of the cone. This relies on a technical non-existence result of certain two-ended expanding solitons. In Section $8$, we describe how the asymptotic cones vary as the initial conditions approach their extreme values. Finally, in Section $9$, we prove that $F$ is proper using the results of Section 8. Then, we define the cohomogeneity one expander degree and prove \Cref{thm:mainthm1} and \Cref{thm:mainthm2}. \\
	
	In Appendix A, we derive the soliton equations under the assumption of the given symmetries. In Appendix B, we explain why the soliton equations have a local solution. In Appendix C, we explain the relationship between Cheeger-Gromov convergence of warped product metrics and convergence of the respective warping functions.
	
	\medskip
	
	\noindent \textbf{Acknowledgements:} The author would like to thank his PhD advisor, Prof. Richard Bamler, for suggesting the problem and for many useful discussions and several pieces of useful advice. We also thank Eric Chen for helpful discussions.
	
	\section{Soliton Equations and Boundary Conditions}
	
	Our goal is to understand gradient expanding Ricci solitons $(M,g,f)$ in $4$ dimensions with an effective isometric action of $\text{SO}(3) \times \text{SO}(2)$. Thus, we look for metrics of the form
	
	\begin{center}
		$g = dr^2 + a(r)^2g_{S^1}+b(r)^2g_{S^2}$
	\end{center}
	
	\noindent with the normalization
	
	\begin{center}
		$\displaystyle \text{Ric}_g + \nabla^2 f + g=0$
	\end{center}
	
	\noindent where $g_{S^1}=d \theta^2$ is the standard metric on the circle $S^1$ and $g_{S^2}$ is the round metric on the sphere $S^2$. We note that both expanding and steady gradient symmetric Ricci solitons on various topologies have been studied in the past, such as in \cite{BDGW15} and \cite{BDW15}, as well as \cite{NW24} in which large families of expanding solitons were constructed, generalizing the $S^2 \times \mathbb{R}^2$ and $S^1 \times\mathbb{R}^3$ solitons constructed in this paper. Using a setup similar to that of \cite{A17}, in Appendix A, we explain how the expanding soliton condition is equivalent to the following equations for $r>0$
	
	\begin{equation}\label{eq:feq}
		\displaystyle f'' = \frac{a''}{a}+2\frac{b''}{b}-1
	\end{equation}
	
	\begin{equation}\label{eq:aeq}
		\displaystyle a'' = -2\frac{a'b'}{b}+a'f'+a
	\end{equation}
	
	\begin{equation}\label{eq:beq}
		\displaystyle b'' = \frac{1-(b')^2}{b}-\frac{a'b'}{a}+b'f'+b
	\end{equation}
	
	\noindent where the smooth functions $f,a,b : [0,\infty) \to \mathbb{R}$ depend only on the coordinate $r$. Considering the soliton equations in this coordinate system as opposed to methods in the aforementioned papers simplifies the (analytic) proofs of completeness and allows us to prove (geometric) curvature estimates in Section 6. These estimates not only aid in proving that each soliton is asymptotically conical, but also gives us a notion of ``uniform $\epsilon$-conicality" (defined in Section 7), which allows us to define the cohomogeneity one expander degree later. 
	
	\medskip
	
	\noindent As $g$ must be a smooth metric at $r=0$, the boundary conditions must be chosen appropriately. It can be seen that some of the boundary conditions are determined by the topology of $M$. We will be particularly interested in the cases of $M$ being diffeomorphic to $S^1 \times \mathbb{R}^3$ and $S^2 \times \mathbb{R}^2$.
	
	\medskip
	
	\begin{lem}
		Suppose $(M, g)$ is a smooth Riemannian manifold with isometric effective action of $\textnormal{SO}(3) \times \textnormal{SO}(2)$, where $g$ is as defined earlier in this section.
		
		\medskip
		
		\noindent For \( M \) to be diffeomorphic to \( S^1 \times \mathbb{R}^3 \), it is necessary and sufficient that 
		\begin{align*}
			a(0) &> 0  & a^{\text{odd}}(0) &= 0 \\
			b^{\text{even}}(0) &= 0       & b'(0) &= 1
		\end{align*}
		
		\noindent Thus, the boundary conditions in this case are
		\begin{align}
			a(0) &= a_0, & a'(0) &= 0 \label{S1R3a} \\
			b(0) &= 0,   & b'(0) &= 1 \label{S1R3b}
		\end{align}
		
		\medskip
		
		\noindent For \( M \) to be diffeomorphic to \( S^2 \times \mathbb{R}^2 \), it is necessary and sufficient that 
		\begin{align*}
			a^{\text{even}}(0) &= 0       & a'(0) &= 1 \\
			b(0) &> 0  & b^{\text{odd}}(0) &= 0
		\end{align*}
		
		\noindent Thus, the boundary conditions in this case are
		\begin{align}
			a(0) &= 0, & a'(0) &= 1 \label{S2R2a} \\
			b(0) &= b_0,   & b'(0) &= 0 \label{S2R2b}
		\end{align}
		
		\noindent where in each case, the corresponding parameter ($a_0$ or $b_0$) is positive. Additionally, every cohomogeneity one gradient expanding Ricci soliton invariant under the standard $\textnormal{SO}(2) \times \textnormal{SO}(3)$ action over either of these topologies arises in this way.
	\end{lem}
	
	\begin{proof}
		The lemma follows from Proposition 1 in the section ``Doubly Warped Products" in Chapter 1 of \cite{Pet}, where it is proven how the boundary conditions above ensure that the topology is as required and that the metric is smooth.
	\end{proof}

	\noindent Next, we impose boundary conditions on $f$ as in \cite{A17}; since the soliton potential is determined only up to a constant, we can choose $f(0)=0$. Additionally, we must have $\lim_{r \to 0}f'(r)=0$ for \cref{eq:aeq,eq:beq} to be satisfied in the limit $r \to 0$, giving the boundary condition $f'(0)=0$.  
	
	\medskip

	\noindent In both cases, putting together the boundary conditions on $a,b$ with those on $f$, \Cref{eq:feq} is degenerate at $r=0$, and a solution can be specified uniquely by imposing a value of $f''(0) \equiv f_0$. In Appendix B, using methods from \cite{A17} and \cite{Buz11}, we show the degeneracy of the equations at $r=0$ and how the boundary conditions above along with a value of $f_0$ ensure the existence of a unique local solution to 
	\cref{eq:feq,eq:aeq,eq:beq}
	
	\section{Soliton Identities}
	
	In this section, we collect some well-known soliton identities which we will combine with equations \cref{eq:feq,eq:aeq,eq:beq} in later sections. Importantly, we show that the scalar curvature at $r=0$ is determined by $f''(0) \equiv f_0$.
	
	\medskip
	
	\noindent Suppose $(M,g, \nabla f)$ is a cohomogeneity one gradient expanding soliton as considered in Section 2. The following identities will be useful in the analysis of the soliton equations.
	
	\begin{equation}\label{eq:trace}
		R+\Delta f + 4 = 0
	\end{equation}
	
	\begin{equation}\label{eq:bianchi}
		R+|\nabla f|^2 + 2f = \text{constant} = R(0)
	\end{equation}
	
	\noindent where $\Delta \equiv \Delta_M$ is the Laplacian of $(M,g)$. \cref{eq:trace} follows from taking the trace of the soliton equation while \cref{eq:bianchi} is an application of the second contracted Bianchi identity and the initial conditions $f(0)=f'(0)=0$. 
	
	\begin{lem}\label{lem:feq} The soliton potential $f$ satisfies $\Delta f - |\nabla f|^2-2f=3f_0$ in the $S^1 \times \mathbb{R}^3$ case and $\Delta f - |\nabla f|^2-2f=2f_0$ in the $S^2 \times \mathbb{R}^2$ case.\end{lem}
	\begin{proof}
		Combining \cref{eq:trace} and \cref{eq:bianchi} gives $\Delta f - |\nabla f|^2+4-2f=-R(0)$. Using \Cref{eq:Laplacian} from Appendix A, rewrite \cref{eq:trace} as
		\begin{equation*}
			R=-4-f''-\left(\frac{a'f'}{a}+2\frac{b'f'}{b}\right)
		\end{equation*}
		
		\noindent In the $S^1 \times \mathbb{R}^3$ case, the boundary conditions \cref{S1R3a,S1R3b} imply that $\frac{a'f'}{a}(0)=0$ while $2\frac{b'f'}{b} \to 2f''(0)$ as $r \to 0$ by L'H\^opital's rule. In the $S^2 \times \mathbb{R}^2$ case,  \cref{S2R2a,S2R2b} imply that $\frac{a'f'}{a}(0)=f''(0)$ while $2\frac{b'f'}{b} \to 0$ as $r \to 0$.  This shows that
		
		\begin{equation}\label{eq:R0}
			R(0) = -3f_0-4 \textnormal{ for } S^1 \times \mathbb{R}^3 \hskip 2cm R(0) = -2f_0-4 \textnormal{ for } S^2 \times \mathbb{R}^2
		\end{equation}
		
		\noindent which gives us the result. 
	\end{proof}
	
	\begin{lem}\label{f_0_bounded} In the case of either topology, for a complete solution of bounded curvature to \cref{eq:feq,eq:aeq,eq:beq} with the appropriate boundary conditions, we must have $f_0 \leqslant 0$.\end{lem}
	\begin{proof}
		Under the assumption of bounded curvature, maximum principle methods (as in, for example, Theorem 2.3 of \cite{BC23}; note the difference in normalizations) allow us to conclude that expanding solitons with the chosen normalization satisfy $R \geqslant -4$. Combining this inequality with \cref{eq:R0} gives us the result in each case.
	\end{proof}
	
	\noindent Given a complete solution to \cref{eq:feq,eq:aeq,eq:beq}, the corresponding metric $g$ as in \Cref{eq:metric} is complete gradient expanding Ricci soliton metric.
	
	\medskip
	
	\noindent A useful quantity is the ratio $P:=\frac{b}{a}$. By calculating using the soliton equations, we see the following:
	
	\begin{equation}\label{eq:Pgrowth}
		\displaystyle P'' = \left(f'-\frac{b'}{b}-2\frac{a'}{a} \right)P'+\frac{1}{b^2}P
	\end{equation}

	\section{Monotonicity Properties}
	
	\noindent From now on, we will assume that $f''(0)$ is non-positive as described in the previous section. In this section, we will deduce the appropriate monotonicity properties of $a,b$ and $f$. We will observe, similar to \cite{A17} (which considered steady solitons as opposed to expanders) that $a$ and $b$ are monotonically increasing and that $f, f', f'' \leqslant 0$. Note that the results in this section do not assume completeness. Denote by $I \subseteq [0,\infty)$ the maximal interval of existence of the solutions to \cref{eq:feq,eq:aeq,eq:beq}; we know that $I$ contains $0$.
	
	\begin{lem}\label{lem:abmon}The functions $a',b'$ are positive on $I-\{0\}$\end{lem}
	
	\begin{proof}
		First, we look at the $S^1 \times \mathbb{R}^3$ case:
		
		\medskip
		
		\noindent Using L'H\^{o}pital's rule on \cref{eq:aeq}, we see that $a''(0)=\frac{a_0}{3}$. As $a'(0)=0$, we see that $a>0$ and $a'>0$ on a small interval of the form $(0,\epsilon)$. Consider the first $r_0 > 0$ (if it exists) with $a'(r_0)=0$; then \cref{eq:aeq} becomes $a'' = a>0$, implying $a'>0$ for a small distance beyond $r_0$.
		
		\medskip
		
		\noindent As $b'(0)=1$, we see that $b>0$ and $b'>0$ on a small interval of the form $(0,\epsilon)$. Consider the first $r_0 > 0$ (if it exists) with $b'(r_0)=0$, then \cref{eq:beq} becomes $b'' = \frac{1}{b}+b>0$, implying $b'>0$ for a small distance beyond $r_0$.
		
		\medskip
		
		\noindent Now, we look at the $S^2 \times \mathbb{R}^2$ case:
		
		\medskip
		
		\noindent As $a'(0)=1$, we see that $a >0$ on a small interval of the form $(0, \epsilon)$. As in the previous case, consider the first $r_0 > 0$ (if it exists) with $a'(r_0)=0$, then \cref{eq:aeq} becomes $a'' = a>0$, implying $a'>0$ for a small distance beyond $r_0$.
		
		\medskip
		
		\noindent Using L'H\^{o}pital's rule on \cref{eq:beq}, we see that $b''(0)=\frac{b_0}{2}$. As $b'(0)=0$, we see that $b>0$ and $b'>0$ on a small interval of the form $(0,\epsilon)$. Consider the first $r_0 > 0$ (if it exists) with $b'(r_0)=0$, then \cref{eq:beq} becomes $b'' = \frac{1}{b} + b>0$, implying $b'>0$ for a small distance beyond $r_0$.
	\end{proof}
	
	\noindent To understand the behavior of $f$, we consider the cases $f''(0) < 0$ and $f''(0)=0$ separately as follows.
	
	\begin{lem}\label{lem:fmon} If $f_0<0$, the functions $f,f',f''$ are negative on $I-\{0\}$ and hence $f$ and $f'$ are monotonically decreasing. \end{lem}
	
	\begin{proof}
		By ~\Cref{lem:feq}, we get 
		\begin{equation}\label{eq:flaplacian}
			\displaystyle f''+\left(\frac{a'}{a}+2\frac{b'}{b} \right)f'-(f')^2-2f=3f''(0)
		\end{equation}
		
		\noindent in the $S^1 \times \mathbb{R}^3$ case (the other case is almost exactly the same, with $3f''(0)$ replaced by $2f''(0)$). As $f''(0) <0$ and $f'(0)=0$, we have $f,f' <0$ on a small interval of the form $(0, \epsilon)$. If there is a point where $f'=0$, let $r_0$ be the first such point. Then, we must have $f(r_0)<0$. Then, at $r_0$, \cref{eq:flaplacian} simplifies to
		
		\begin{equation*}
			f''(r_0) = 3f''(0)+2f(r_0)<0
		\end{equation*}
		
		\noindent This shows that $f$ is monotonically decreasing and $f' <0$ for $r>0$.
		
		\medskip
		
		\noindent As $f_0 <0$, we know that $f''<0$ on an interval of the form $(0, \epsilon)$. Differentiating \cref{eq:flaplacian}, we get
		
		\begin{equation*}
			\displaystyle f'''=2f'+2f'f''-\left(\frac{a'}{a}+2\frac{b'}{b} \right)'f'-\left(\frac{a'}{a}+2\frac{b'}{b} \right)f''
		\end{equation*}
		
		\noindent Applying equations \cref{eq:feq,eq:aeq,eq:beq}, we see that 
		
		\begin{equation*}
			\displaystyle \left(\frac{a'}{a}+2\frac{b'}{b} \right)'= f''+1-\left(\left(\frac{a'}{a}\right)^2+2\left(\frac{b'}{b}\right)^2 \right)
		\end{equation*}
		
		\noindent Thus, we see that 
		
		\begin{equation*}
			\displaystyle f'''=f'+f'f''+\left(\left(\frac{a'}{a}\right)^2+2\left(\frac{b'}{b}\right)^2 \right)f'-\left(\frac{a'}{a}+2\frac{b'}{b} \right)f''
		\end{equation*}
		
		\noindent This implies that at a point where $f''=0$, we have
		
		\begin{equation}
			\displaystyle f'''=\left(1+\left(\frac{a'}{a}\right)^2+2\left(\frac{b'}{b}\right)^2 \right)f' < 0
		\end{equation}
		
		\noindent which shows that $f''<0$.
		
	\end{proof}
	
	\begin{lem}\label{lem:fzero} If $f_0=0$ then $f \equiv 0$. \end{lem}
	
	\begin{proof}
		Since the solution to equations \cref{eq:feq,eq:aeq,eq:beq} depends continuously on the parameter $f''(0)$ (by the results of Appendix B), we know by \Cref{lem:fmon} that $f,f',f'' \leqslant 0$ when $f''(0)=0$. Equation \cref{eq:flaplacian} becomes 
		
		\begin{equation}\label{eq:flaplacianzero}
			\displaystyle f''=-\left(\frac{a'}{a}+2\frac{b'}{b} \right)f'+(f')^2+2f
		\end{equation}
		
		\noindent Using \cref{eq:flaplacianzero}, by standard theory of ordinary differential equations, it is sufficient to show that $f$ is identically zero in a neighborhood of $0$ to conclude that $f\equiv0$ on $\mathbb{R}$. Suppose by way of contradiction that this is not the case -- then, there is an interval $(0,\epsilon)$ on which $f,f'<0$. Then, we can rewrite \cref{eq:flaplacianzero} as 
		
		\begin{equation}\label{eq:flaplacianzerostep}
			\displaystyle f''=-\left(\left(\frac{a'}{a}+2\frac{b'}{b}\right) +f'+2\frac{f}{f'} \right)f'
		\end{equation}
		
		\noindent We notice that $f'$ $ \to 0$ as $r \to 0$ and that $\left(\frac{a'}{a}+2\frac{b'}{b}\right) \to \infty$ as $r \to 0$ (since $b(0)=0$ and $b'(0)=1$ in the $S^1 \times \mathbb{R}^3$ case, and $a(0)=0$ and $a'(0)=1$ in the $S^2 \times \mathbb{R}^2$ case). Additionally,
		
		\begin{center}
			$\displaystyle \left| \frac{f(r)}{f'(r)} \right| = \left| \frac{\int_{0}^{r}f'(s)ds}{f'(r)} \right| \leqslant \left| \frac{rf'(r)}{f'(r)} \right| =r$
		\end{center}

		\noindent where the inequality follows since $|f'(s)| \leqslant |f'(r)|$ for $s<r$, as $f'' \leqslant 0$. Thus, $\frac{f'}{f} \to 0$ as $r \to 0$.
		
		\medskip
		
		\noindent  This implies that the quantity in the parentheses in \cref{eq:flaplacianzerostep} is positive at a point in $(0,\epsilon)$; this is a contradiction, as this would imply that $f''>0$ at that point.
	\end{proof}	
	
	\noindent We have one more monotonicity result, for the quantity $P=\frac{b}{a}$.
	
	\begin{lem}\label{lem:Pgrowth}
		Suppose there exists an $r_0 > 0$ with $P'(r_0) \geqslant 0$. Then, $P'(r) \geqslant 0$ for all $r \geqslant r_0$.
	\end{lem}
	
	\begin{proof}
		Suppose $r_0$ exists and $r_1 >r_0$ is the first point beyond $r_0$ with $P'(r_1)=0$. Then, by equation \Cref{eq:Pgrowth} we see that $P''(r_1) = \frac{P}{b^2} >0$, implying that $P'$ remains nonnegative. As a consequence, if $P$ is ever increasing, it remains increasing.
	\end{proof}
	
	\section{Completeness of Solitons}
	
	The main results of this section are the following theorem and its corollary which show that there exists a $2$-parameter family of complete cohomogeneity one gradient expanding Ricci solitons in the case of each topology. Note that our assumption of the initial condition $f''(0) \equiv f_0$ being negative is sufficient for a complete solution (and necessary for bounded curvature, as explained in Section 3).

	\begin{thm}\label{thm:2pf}For either set of boundary conditions \cref{S1R3a,S1R3b} or \cref{S2R2a,S2R2b}, if $f_0 < 0$ and $a_0, b_0>0$, there exists a unique complete solution $f,a,b : [0,\infty) \to \mathbb{R}$ to the soliton equations \cref{eq:feq,eq:aeq,eq:beq}. \end{thm}
	
	\noindent We note that \Cref{thm:2pf} is a special case of more general theorems in \cite{NW24}, \cite{Win21}, and \cite{BDGW15}. For the sake of completeness, we provide an alternate proof in this special case. We also remark that we use some of the estimates in the proof from this section in future sections.
	
	\begin{cor}\label{cor:2pf}
		Let \( M \) be a smooth manifold diffeomorphic to \( S^1 \times \mathbb{R}^3 \) or \( S^2 \times \mathbb{R}^2 \). Then there exists a two-parameter family of complete gradient expanding Ricci solitons on \( M \) which are cohomogeneity one and invariant under the standard action of \(\mathrm{SO}(3) \times \mathrm{SO}(2)\).
	\end{cor}
	
	\begin{proof}
		Completeness follows immediately from the theorem. In the $S^1 \times \mathbb{R}^3$ case, the parameters are $a_0$ and $f_0$, while in the $S^2 \times \mathbb{R}^2$ case, the parameters are $b_0$ and $f_0$.
	\end{proof}
	
	\begin{proof}[Proof of \Cref{thm:2pf}]We know by Appendix B that \cref{eq:feq,eq:aeq,eq:beq}  have a local solution. Then, we exhibit growth bounds on $a$ and $b$ that allow us to extend them indefinitely. Then, we repeat the process for $f$.
		
		\begin{lem}\label{lem:abextend}	Suppose $(M,g,\nabla f)$ is a cohomogeneity one gradient expanding Ricci soliton diffeomorphic to \( S^1 \times \mathbb{R}^3 \) or \( S^2 \times \mathbb{R}^2 \). Under the hypothesis of \Cref{thm:2pf}, $a,b$ and $a',b'$ remain bounded on the maximal interval of existence, which is $[0, \infty)$. \end{lem}
		\begin{proof}
			By the existence of the local solution, the maximal interval of existence of the solutions to \cref{eq:feq,eq:aeq,eq:beq} contains an interval of the form $(0,2\epsilon)$ for some $\epsilon>0$ and by Lemmas\autoref{lem:abmon} and\autoref{lem:fmon}, we know that $a',b'>0$ and $f'<0$. Thus, $a$ and $b$ are positive on  $(0,2\epsilon]$. Then, we can rewrite \cref{eq:aeq} and \cref{eq:beq} as
			
			\begin{equation}\label{eq:agrowth}
				a'' < a
			\end{equation}
			
			\begin{equation}\label{eq:bgrowth}
				b'' < \frac{1}{b}+b
			\end{equation}
			
			\noindent Multiplying \cref{eq:agrowth} by $a'$ on both sides and integrating gives us 
			
			\begin{equation}\label{eq:apgrowth}
				(a')^2 \leqslant a^2 - a(0)^2 +(a'(0))^2
			\end{equation}
			
			\noindent In the $S^1 \times \mathbb{R}^3$ case, using \cref{S1R3a} and \cref{S1R3b}, \Cref{eq:apgrowth} implies that $a' \leqslant a$ on $[\epsilon,\infty)$. Integrating this shows that $a$ is bounded by an exponential function.
			
			\medskip
			
			\noindent In the $S^2 \times \mathbb{R}^2$ case, using \cref{S2R2a} and \cref{S2R2b}, \Cref{eq:apgrowth} implies that $(a')^2 \leqslant a^2+1$ on $[\epsilon,\infty)$. If $a$ is globally bounded by $1$, then we are done. If not, we have $a(r_0) > 1$ for some $r_0 > \epsilon$; thus, by the monotonicity of $a$, we have $1 \leqslant a^2$ on $[r_0, \infty)$, giving us the bound $(a')^2 \leqslant 2a^2$. Integrating this shows that $a$ is bounded by an exponential function.
			
			\medskip
			
			\noindent For $b$, as $b'>0$, we know that for $r>r_0$, $b>b(\epsilon)=C$ for some constant $C>0$ by \Cref{lem:abmon}. This allows us to rewrite \Cref{eq:bgrowth} as
			
			\begin{equation}\label{eq:bpgrowth}
				b'' < C+b
			\end{equation}
			
			\noindent on the interval $[\epsilon, \infty)$. Now, in both cases, similar analysis shows that $b$ is bounded by an exponential function whose value and derivative at $r=\epsilon$ match those of $b$. Thus, $a$ and $b$ do not blow up at finite $r$ and these functions can be extended to $[0,\infty)$ by standard ODE theory.
		\end{proof}
		
		\begin{lem}\label{lem:fextend}Under the hypothesis of \Cref{thm:2pf}, $f$ and $f'$ remain bounded on $[0,\infty)$\end{lem}
		
		\begin{proof}
			We prove this in the $S^1 \times \mathbb{R}^3$ case; the $S^2 \times \mathbb{R}^2$ case is identical, except for changing the $3f''(0)$ term to $2f''(0)$. Applying Lemmas\autoref{lem:abmon} and\autoref{lem:fmon} to equation \cref{eq:flaplacian}, we see that 
			
			\begin{align*}
				0<&(f')^2 \\
				=&-(3f''(0)+2f)+\left(\frac{a'}{a}+2\frac{b'}{b} \right)f'+f'' \\
				<&-(3f''(0)+2f)
			\end{align*}
			
			\noindent This yields the inequality $f' > - \sqrt{-(3f''(0)+2f)}$. Solving the inequality shows that $|f|$ is bounded by a quadratic function. 
			
			\medskip
			
			\noindent More explicitly, suppose that $\tilde{f}$ satisfies the ODE corresponding to the differential inequality with $\tilde{f}(0)=f(0)$. Then, $\tilde{f}(r) = -\frac{r}{2}(2\sqrt{-3f''(0)}+r)$, and $\tilde{f}$ is a lower bound for $f$.  Hence, by standard ODE theory as usual, $f$ can be extended smoothly to $[0,\infty)$.	
		\end{proof}
		
		\noindent Uniqueness follows from the uniqueness of the local solution in Appendix B and standard ODE theory. This concludes the proof of \Cref{thm:2pf}.
	\end{proof}
	
	\section{Asymptotics}
	
	\noindent Suppose $a,b,f$ satisfy the soliton equations \cref{eq:feq,eq:aeq,eq:beq}. The main result of this section is that the complete expanding Ricci solitons constructed in Section 5 are asymptotic to cones over the link $S^2 \times S^1$. Several technical lemmas analyzing the soliton equations will be needed before we conclude the result. We continue to assume that $f''(0) \equiv f_0 < 0$ to ensure that our solitons are complete.
	
	\medskip

	\noindent In this section, the proofs will be carried for the soliton equations in the $S^1 \times \mathbb{R}^3$ case; the other case is nearly identical, with the difference being the term $2f''(0)$ appearing instead of $3f''(0)$.
	
	\medskip
	
	\noindent As we expect the expanding solitons to be asymptotic to cones, the quantities $\frac{a'}{a}$ and $\frac{b'}{b}$ should decay like $\frac{1}{r}$ to $0$. We first show that these quantities are bounded by constants in a region near infinity. 
	
	\begin{lem}\label{lem:absimplebound} Suppose $a,b,f$ satisfy equations \cref{eq:feq,eq:aeq,eq:beq}, with boundary conditions either \cref{S1R3a,S1R3b} or \cref{S2R2a,S2R2b}, depending on the topology. Then, there exists a $C > 0$ such that $a'<Ca$ and $b'<Cb$ on $[1,\infty)$, where $C$ depends continuously on the initial conditions (so $C \equiv C(a_0, f_0)$ in the $S^1 \times \mathbb{R}^3$ case and $C \equiv C(b_0, f_0)$ in the $S^2 \times \mathbb{R}^2$ case) \end{lem}
	\begin{proof}
		We will carry out the proof for $b$; the proof for $a$ is almost identical to \Cref{eq:apgrowth}. By \Cref{eq:bgrowth}, we know that
		
		\begin{equation*}
			\displaystyle \frac{b''}{b} < \frac{1}{b^2}+1
		\end{equation*}
		
		\noindent which implies that
		
		\begin{align*}
			\left( \frac{b'}{b}\right)' &=  \frac{b''}{b} -\left(\frac{b'}{b}\right)^2 \\
			&< \frac{1}{b^2}+1-\left(\frac{b'}{b}\right)^2 \\ 
			&\leqslant \frac{1}{b(1)^2}+1-\left(\frac{b'}{b}\right)^2
		\end{align*}
		
		\noindent on the interval $[1,\infty)$ (where the last step follows since $b > b(1)$, as $b$ is increasing), giving us the inequality
		
		\begin{equation}
			\displaystyle \left( \frac{b'}{b}\right)'  < C'-\left(\frac{b'}{b}\right)^2
		\end{equation}
		
		\noindent for the constant $C' = \frac{1}{b(1)^2}+1$. Thus, the quantity $\frac{b'}{b}$ satisfies the inequality $u' +u^2<C'$. Thus, $\frac{b'}{b}$ is bounded by the solution to the initial value problem
		
		\begin{center}
			$u'(x)+u(x)^2=C' $
			
			\smallskip
			
			$u(1)=\frac{b'}{b}(1)$
		\end{center}
		
		\noindent This IVP can be solved exactly and the solution $u$ is asymptotic to $\sqrt{C'}$. If $u(1)<\sqrt{C'}$, then $u$ increases and becomes asymptotic to $\sqrt{C'}$, and if $u(1)>\sqrt{C'}$, then $u$ decreases to $\sqrt{C'}$. As $\frac{b'(r)}{b(r)}\leqslant u(r)$ on $[1,\infty)$, any $C$ greater than $\text{max}(\sqrt{C'},\frac{b'}{b}(1))$ makes the lemma true. Clearly, such a $C$ can be chosen continuously in the initial conditions, since $b(1)$ and $b'(1)$ vary smoothly in the initial conditions as in the hypothesis of the lemma.
	\end{proof}

	\noindent The next lemma is an improvement of \Cref{lem:fmon}. It shows that $f''$ is bounded from above by a negative constant.
	
	\begin{lem}\label{lem:fppbound} Suppose $a,b,f$ satisfy equations \cref{eq:feq,eq:aeq,eq:beq}, with boundary conditions either \cref{S1R3a,S1R3b} or \cref{S2R2a,S2R2b}, depending on the topology. Then, there exists $\epsilon > 0$, depending continuously on the initial conditions (so $\epsilon \equiv \epsilon(a_0, f_0)$ or $\epsilon \equiv \epsilon(b_0, f_0)$ depending on the topology), such that for $r>1$, we have
		\begin{center}
			$f'' \leqslant  -\epsilon$
		\end{center}
	\end{lem}
	
	\begin{proof}
		
		Recall from the proof of \Cref{lem:fmon} that we have
		
		\begin{equation*}
			\displaystyle f'''=f'+f'f''+\left(\left(\frac{a'}{a}\right)^2+2\left(\frac{b'}{b}\right)^2 \right)f'-\left(\frac{a'}{a}+2\frac{b'}{b} \right)f''
		\end{equation*}
		
		\noindent For $r>1$, we see that
		
		\begin{equation*}
			\displaystyle f'''(r)<f'(r)(1+f''(r))-3Cf''(r)
		\end{equation*}
		
		\noindent by the monotonicity properties and \Cref{lem:absimplebound}. Now, choose $0 < \epsilon < 1$ so that the quantity $f'(1)(1-\epsilon)+3C\epsilon$ is negative. This is possible as this quantity is equal to $f'(1)$ (which is negative) for $\epsilon =0$, so there must exist a positive $\epsilon$ satisfying the condition. Further adjust $\epsilon$ if needed so that $f''(1) < -\epsilon$. Then, for $r$ in a small interval $[1,1+\delta)$, we have $f'' < -\epsilon$.
		
		\medskip
		
		\noindent Then, we see that if there exists a point $r>1$ with $f''(r)=-\epsilon$, then
		\begin{align*}
			f'''(r)&<f'(r)(1+f''(r))-3Cf''(r) \\
			&=  f'(r)(1-\epsilon)+3C\epsilon \\
			&< f'(1)(1-\epsilon)+3C\epsilon \\ 
			&<0
		\end{align*}
		
		\noindent where the third line follows by the monotonicity of $f'$ and the last line by the choice of $\epsilon$. Thus, $f'''(r) <0$ at any point where $f''(r) = -\epsilon$, implying that $f'' < -\epsilon$ is a preserved condition and thus holds on $[1,\infty)$, implying the statement of the lemma. 
		
		\medskip
		
		\noindent Note that the choice of $\epsilon$ depends on $C$ from \Cref{lem:absimplebound} and $f'(1)$ and $f''(1)$, which together depend continuously on the initial conditions $f''(0)$ and $a_0$ or $b_0$.	
	\end{proof}

	\noindent Now, we prove a lower bound on $f'$. Note that the geometric meaning of this bound is that $|\nabla f|$ has at most linear growth.
	
	\begin{lem}\label{lem:fpbound}
		The inequality $f'(r) > -(r+C_1)$ holds on $[1,\infty)$, where $C_1 \in \mathbb{R}$ is a real constant depending continuously on the value of $f_0$. In fact, we can choose $C_1 = \sqrt{-3f_0}$
	\end{lem}
	
	\begin{proof}
		We prove this in the $S^1 \times \mathbb{R}^3$ case; the other case is nearly identical. By the proof of \Cref{lem:fextend}, we have the inequality $f'(r) > -\sqrt{-(3f''(0)+2f(r))}$. Consider the solution $\tilde{f}$ to the corresponding ODE $\tilde{f}'(r) = -\sqrt{-(3f''(0)+2\tilde{f}(r))}$ with $f(1)=\tilde{f}(1) < 0$. Then, we have $f \geqslant \tilde{f}$ on $[1,\infty)$. This leads to the chain of inequalities
		
		\begin{align*}
			f'(r) &> -\sqrt{-(3f''(0)+2f(r))} \\
			&> -\sqrt{-(3f''(0)+2\tilde{f}(r))} \\
			&= \tilde{f}'(r)
		\end{align*}
		
		\noindent The ODE for $\tilde{f}$ can be solved explicitly as in \Cref{lem:fextend}, with
		
		\begin{center}
			$\tilde{f}(r) = -\dfrac{r}{2}\left(2\sqrt{-3f''(0)}+r \right) \hskip 2cm \tilde{f}'(r) = -(r+\sqrt{-3f''(0)})$ 
		\end{center}

		\noindent and by substituting $\tilde{f}(r)$ into the inequality above, we see that $f'(r) >-(r+C_1)$ for some real constant $C_1$, as required.
	\end{proof}
	
	\noindent We can put the previous two lemmas together to control the growth rate of $f$ in the following way:
	
	\begin{thm}\label{thm:fpbound}
		Suppose $a,b,f$ satisfy equations \cref{eq:feq,eq:aeq,eq:beq}, with boundary conditions \cref{S1R3a,S1R3b} or \cref{S2R2a,S2R2b}. Then, the soliton vector field $f'(r)$ satisfies bounds of the following form:
		
		\begin{equation*}
			-(r+C_1) \leqslant f'(r) \leqslant -\epsilon (r-1)
		\end{equation*}
		
		\noindent where $C_1$ is a constant depending on $f_0$ and $\epsilon \equiv \epsilon(a_0, f_0)$ or $\epsilon \equiv \epsilon(b_0, f_0)$ is a positive constant depending continuously on the initial conditions.
	\end{thm}
	
	\begin{proof}
		The lower bound is \Cref{lem:fpbound}, while the upper bound follows by integrating the bound in \Cref{lem:fppbound} on $[1,r]$ and using the fact that $f'(1) < 0$ by \Cref{lem:fmon}.
	\end{proof}
	
	\medskip
	
	\noindent We will use Theorem 6.4 and further bounds on $f'(r)$ to provide growth bounds on $a$ and $b$. For this, we will need the following ODE comparison result:
	
	\begin{lem}\label{lem:ODE}
		Suppose there exist functions \( v_1 \), \( v_2 \) $: \mathbb{R}^+ \to \mathbb{R}$  satisfying the differential inequalities
		\[
		v_1''(r) \leqslant -(r+C)v_1'(r) + v_1(r), \quad \quad v_2''(r) \geqslant -(r+C)v_2'(r) + v_2(r)
		\]
		for \( r \geq r_0 \), with initial conditions \( v_1(r_0) = v_2(r_0) = \alpha \) and \( v_1'(r_0) = v_2'(r_0) = \alpha' \) and $C$ is a positive real number. Then the following hold:
		\begin{enumerate}
			\item \( v_1(r) \leqslant c_1(r - r_0) + \alpha \) for all \( r \in [r_0, \infty) \), for some positive constant \( c_1 \) which can be chosen uniformly in the constants \( C \), \( \alpha \), and \( \alpha' \). In addition, $v_1' \leqslant c_1$ on this interval.
			\item \( v_2(r) \geqslant C_1(r - r_0) + \alpha \) for all \( r \in [r_0, \infty) \), for some positive constant \( C_1 \) which can be chosen uniformly in the constants \( C \), \( \alpha \), and \( \alpha' \). In addition, $v_2' \geqslant C_1$ on this interval.
		\end{enumerate}
	\end{lem}
	
	\begin{proof}

		\noindent 1. For any $c_1 > 0$, consider the function $w_1 : [r_0, \infty) \to \mathbb{R}$ given by $w_1(r) = c_1(r-r_0)+\alpha$. The value of $c_1$ will be chosen below. Then, we see that
		\begin{align*}
			(v_1 - w_1)'' &= v_1'' \\
			&\leqslant -(r + C)v_1' + v_1 \\
			&\leqslant -(r + C)v_1' + v_1 + c_1(C + r_0) - \alpha \\
			&= -(r + C)(v_1 - w_1)' + (v_1 - w_1)
		\end{align*}

		\noindent where the last inequality is true for $c_1 \geqslant {\alpha}/(C+r_0)$. Additionally, note that $\alpha = v_1(r_0) = w_1(r_0)$, and that $(v_1-w_1)'(r_0) = \alpha'-c_1$. 
		
		\medskip
		
		\noindent Now, choose $c_1 > \text{max}\{{\alpha}/(C+r_0), \alpha'\}$. Then, on $[r_0, \infty)$ we have the inequality
		
		\begin{center}
			$(v_1-w_1)'' \leqslant -(r+C)(v_1-w_1)'+(v_1-w_1)$
		\end{center}
		
		\noindent with $(v_1-w_1)(r_0) = 0$ and $(v_1-w_1)'(r_0) < 0$. Thus, $v_1 \leqslant w_1$ and $v'_1 \leqslant w'_1 = c_1$ on an interval of the form $[r_0, r_0 + \epsilon]$ for some $\epsilon>0$. Let $r^* > r_0$ be the first point with $v'_1 (r^*) = w'_1(r^*)$, if any such points exist; then, by the computation above, we have $(v_1 - w_1)''(r^*) \leqslant 0$. Thus, we see that $v_1 \leqslant w_1$ and $v_1' \leqslant w_1'$ on $[r_0, \infty)$. Clearly, the value of $c_1$ can be chosen uniformly in $\alpha, \alpha'$ and $C$; for example, we may choose 
		$c_1 = \text{max}\{{\alpha}/(C+r_0), \alpha'\} + 1$
		
		\medskip
		
		\noindent 2. For any $C_1 > 0$, consider the function $w_2 : [r_0, \infty) \to \mathbb{R}$ given by $w_2(r) = C_1(r-r_0)+\alpha$. As in Part 1, the value of $C_1$ will be chosen below. Then, we see that
		\begin{align*}
			(v_2 - w_2)'' &= v_2'' \\
			&\geqslant -(r + C)v_2' + v_2 \\
			&\geqslant -(r + C)v_2' + v_2 + C_1(C + r_0) - \alpha \\
			&= -(r + C)(v_2 - w_2)' + (v_2 - w_2)
		\end{align*}

		\noindent where the last inequality is true for $C_1 \leqslant {\alpha}/(C+r_0)$. Additionally, note that $\alpha = v_2(r_0) = w_2(r_0)$, and that $(v_2-w_2)'(r_0) = \alpha' - C_1$. 
		
		\medskip
		
		\noindent Now, choose $C_1 < \text{min}\{{\alpha}/(C+r_0), \alpha'\}$. Then, on $[r_0, \infty)$ we have the inequality
		
		\begin{center}
			$(v_2-w_2)'' \geqslant -(r+C)(v_2-w_2)'+(v_2-w_2)$
		\end{center}
		
		\noindent with $(v_2-w_2)(r_0) = 0$ and $(v_2-w_2)'(r_0) > 0$. Thus, as before, we see that $v_2 \geqslant w_2$ and $v_2' \geqslant w_2'$ on $[r_0, \infty)$. Clearly, the value of $C_1$ can be chosen uniformly in $\alpha, \alpha'$ and $C$; for example, set $C_1 = \displaystyle \frac{\text{min}\{{\alpha}/(C+r_0), \alpha'\}}{2}$
	\end{proof}
	
	\begin{lem}\label{lem:ablinbelow}
		Consider \cref{eq:feq,eq:aeq,eq:beq}, with initial conditions \cref{S1R3a,S1R3b} or \cref{S2R2a,S2R2b}. Then, on $[1, \infty)$, we have bounds of the form $a(r), b(r) \geqslant c(r-1)$, and the constant $c \equiv c(a_0,f_0)$ or $c \equiv c(b_0,f_0)$ can be chosen uniformly in the initial conditions.
	\end{lem}
	
	\begin{proof}
		We will prove the result for $a$ in the $S^1 \times \mathbb{R}^3$ case; the proofs for the other case are similar. Suppose we have an initial condition $(a_0, f_0)$ and a compact set $F$ containing it such that $F$ avoids the boundary of the space of initial conditions (that is, both coordinates are nonzero for any element of $F$).
		
		\medskip
		
		\noindent \textbf{Step 1:} Consider the soliton equation \cref{eq:aeq} rewritten as $a'' = a+a'(-2\frac{b'}{b}+f')$. Using Lemmas\autoref{lem:absimplebound} and\autoref{lem:fpbound}, that  $\frac{b'}{b}<C$ and $f' > - (r+C_1)$, and that $a'>0$ by \Cref{lem:abmon}, we can extract the following inequality on $[1, \infty)$:
		
		\begin{center}
			$ a''(r) \geqslant a(r) - (r+\bar{C})a'(r)$
		\end{center}
		
		\medskip
		
		\noindent  where $\bar{C}:=C_1+2C$ is a positive constant, and $a$ has initial conditions $a(1)$ and $a'(1)$. 
		
		\medskip 
		
		\noindent \textbf{Step 2:} By continuous dependence in the initial conditions of solutions to \cref{eq:feq,eq:aeq,eq:beq}, we notice that the values $a(1)$ and $a'(1)$ are bounded from above and below (with nonzero lower bounds) by our choice of compact set $F$. Additionally, $\bar{C}$ can be chosen uniformly over $F$ as well by the uniformity of $C_1$ and $C$. Then, by Part 1 of Lemma 6.5, we have a bound $a(r) \geqslant c_1(r-1) + a(1)$ for some $c_1 >0$ which can be chosen uniformly over $F$ and the constant $\bar{C}$. Thus, all $a$ with initial conditions in $K$ are bounded from below by a linear function. The proof for $b$ is similar, following from equation \cref{eq:beq}.
	\end{proof}

	\medskip
	
	\noindent Now that we understand how $f'$ behaves for large $r$, we can calculate the asymptotics of $a$ and $b$. We begin with lemmas describing the boundedness of the curvature and the rate of decay of the curvature near infinity. Before we do this, we prove a point-picking lemma that we will use multiple times in this paper. The lemma and proof are taken from \cite{RFLN}.
	
	\begin{lem}[Point Picking]\label{lem:pp}
		Let $M$ be a complete manifold (with or without boundary) with $f: M \to (0,\infty)$ continuous, $x \in M$, and $d>0$. Then, there is a $y \in B(x, 2d f(x)^{-1/2})$ such that $f(y) \geqslant f(x)$ and $f \leqslant 4f(y)$ on $B(y, d f(y)^{-1/2})$
	\end{lem}
	
	\begin{proof}
		Set $y_0=x$. If $y=y_0$ satisfies the required conditions, the proof concludes here. Else there exists $y_1 \in B(y_0, d/\sqrt{f(y_0)})$ such that $f(y_1) > 4 f(y_0)$. If $y_1$ satisfies the required conditions, the proof concludes here. Otherwise, repeat this process to produce a sequence $\{y_j\}$. By repeatedly applying the triangle inequality, we obtain
		\[
		d(y_j, x) = d(y_j, y_0) \leqslant \frac{d}{\sqrt{f(y_0)}} \left(1 + \frac{1}{2} + \cdots + \frac{1}{2^{j-1}} \right) < \frac{2d}{\sqrt{f(x)}}
		\]
		
		\noindent Since the closure of $B(x, 2d f(x)^{-1/2})$ is compact, we get an upper bound on $f$ on this ball, so this process has to terminate. Thus, there exists a sufficiently large $j \in \mathbb{N}$ so that $y=y_j$ satisfies the required conditions.
	\end{proof}
	
	\noindent In the following lemma, we use geometric methods to prove estimates on the curvature. In using geometric convergence methods, we use the fact that since our metrics are warped products, the limit metrics are also warped products. The precise statement is proven in Appendix C.
	
	\begin{lem}\label{lem:Rmbound1}
		Consider a complete cohomogeneity one gradient expanding Ricci soliton $(M,g,\nabla f)$ over either $S^1 \times \mathbb{R}^3$ or $S^2 \times \mathbb{R}^2$, as in the setup of this paper. Then, there exists a constant $C>0$ such that $|\textnormal{Rm}|_g \leqslant C$. Moreover, $C$ can be chosen uniformly in the initial conditions (so $C \equiv C(a_0, f_0)$ or $C \equiv C(b_0, f_0)$, depending on the topology)
	\end{lem} 
	
	\begin{proof}
		We will carry out the proof in the $S^1 \times \mathbb{R}^3$ case; the other case is almost identical. Suppose that the statement of the lemma is not true. Then, there exists a sequence of solitons $(M,g_i, \nabla f_i)$ with initial conditions $(a^i_0, f^i_0)$ lying in a compact set $F$ of $\mathbb{R}^+ \times \mathbb{R}^-$ and points $p_i \in (M,g_i)$ with $r(p_i) = r_i$ and $|\text{Rm}|_{g_i}(p_i) := Q_i \to \infty$. Additionally, for any sequence $\{D_i\} \to \infty$, by the previous lemma, we can assume that the $p_i$ are chosen so that $|\text{Rm}|_{g_i} \leqslant 4Q_i$ on $B_{g_i}(p_i, D_i/ \sqrt{Q_i} )$.
		
		\medskip
		
		\noindent First, we claim that $\{r_i\}$ must be unbounded. Since the initial conditions $(a^i_0, f^i_0)$ are bounded for every soliton in the sequence, we know that $a_i$ and $b_i$ and their derivatives remain uniformly bounded on any interval of the form $[0,R]$ for any $R>0$, by the smooth dependence of the solutions of ODEs on initial conditions. As the sectional curvatures of $(M,g_i)$ are smooth functions of $a_i$ and $b_i$ and their derivatives, we see that they must also remain bounded on points whose distance from the singular orbit lies in $[0,R]$. This implies that since $Q_i \to \infty$, the sequence $r_i$ must be unbounded.
		
		\medskip
		
		\noindent Now, rescale $g_i$ by $Q_i$ to get $\tilde{g}_i = Q_i g_i$, and consider the pointed sequence of manifolds $(M,\tilde{g}_i, \nabla f_i, p_i)$. The rescaled manifolds satisfy the equation
		
		\begin{equation*}
			\text{Ric}_{\tilde{g}_i} + \nabla^2 f_i + \frac{1}{Q_i}\tilde{g}_i = 0
		\end{equation*}
		
		\noindent where we now have the bound $|\text{Rm}|_{\tilde{g}_i} \leqslant 4$ on $B_{\tilde{g}_i}(p_i, D_i )$. Now, since $D_i \to \infty$, for any $D>0$, we have the bound $|\text{Rm}|_{\tilde{g}_i} \leqslant C(D)$ on $B_{\tilde{g}_i}(p_i, D )$ (for all $i$) for some constant $C(D)$ depending on $D$.
		
		\medskip

		\noindent By the equation above, this implies that $|\nabla^2 f_i|_{\tilde{g}_i} \leqslant C(D)$ on $B_{\tilde{g}_i}(p_i, D)$. We also have volume lower bounds of small $r$-balls at $p_i$, since the functions $a_i$ and $b_i$ are increasing. By Shi's estimates applied to the associated Ricci flows, we also have bounds on $B_{\tilde{g}_i}(p_i, D)$ on derivatives of the curvature of the form $|\nabla^k \text{Rm}|_{\tilde{g}_i} \leqslant C_k(D)$ for $k \geqslant 1$ and for all $i$; these bounds provide bounds on higher derivatives of $f_i$ as well.
		
		\medskip
		
		\noindent Now, first we consider the case where (up to a subsequence) $|\nabla f|_{\tilde{g}_i}(p_i)$ becomes unbounded. Set 
		
		\begin{center}
			$\displaystyle \tilde{f}_i := \frac{f_i-f_i(p_i)}{|\nabla f|_{\tilde{g}_i}(p_i)}$.
		\end{center}

		\medskip
		
		\noindent Then, $| \nabla^2 \tilde{f}_i |_{\tilde{g}_i}$ converges to $0$, so we have smooth pointed Cheeger-Gromov convergence of a subsequence of $(M,g_i, \nabla \tilde{f}_i, p_i)$ to a smooth non-flat Riemannian manifold $(M_\infty, g_\infty, p_\infty)$ with a smooth function $f_\infty$ satisfying $\nabla^2 f_\infty=0$ and $|\nabla f_\infty|(p_\infty) = 1$. Thus, $\nabla f_\infty$ is a parallel vector field, which implies the splitting of $(M_\infty, g_\infty)$. In addition, since $a_i(r_i)$ and $b_i(r_i)$ tend to infinity at least linearly in $r_i$ by \Cref{lem:ablinbelow}, and $Q_i >1$ for large $i$, we see that $M_\infty$ is diffeomorphic to $\mathbb{R}^4$ and carries a doubly warped product metric $g_\infty = dr^2 + a_\infty(r)^2g_{\mathbb{R}}+b_\infty(r)^2g_{\mathbb{R}^2}$ over $\mathbb{R} \times \mathbb{R}^2$, by \Cref{lem:CG} in Appendix C. As $\nabla^2 f_\infty = 0$, we see that the functions $a_\infty$ and $b_\infty$ are constant, implying that the limit is isometric to Euclidean space. However, $|\text{Rm}|(p_\infty)$ is nonzero, which is a contradiction.
		
		\medskip
		
		\noindent Now, consider the case where $|\nabla f|_{\tilde{g}_i}(p_i)$ remains bounded. Then, we have smooth pointed Cheeger-Gromov convergence of a subsequence of $(M,g_i, \nabla \tilde{f}_i, p_i)$ to a smooth non-flat steady Ricci soliton $(M_\infty, g_\infty, \nabla f_\infty, p_\infty)$. As before, we see that $M_\infty$ is diffeomorphic to $\mathbb{R}^4$ and carries a doubly warped product metric over $\mathbb{R} \times \mathbb{R}^2$. Taking the quotient by $\mathbb{Z}$ and $\mathbb{Z}^2$ so that the orbits are compact, we have a steady Ricci soliton with $2$ ends (since $r_i \to \infty$), which, according to the results of \cite{MW11} must split as the product of $\mathbb{R}$ with a compact Ricci-flat $3$-manifold $N$, which must be flat. This implies that $M_\infty$ is flat, which is a contradiction.
		
		\medskip
		
		\noindent	Thus, such a sequence $\{r_i\}$ cannot exist, so the curvatures of the solitons must be uniformly bounded in terms of the initial conditions. 	\end{proof}
	
	\medskip
	
	\noindent Now, we improve the curvature bound from the previous lemma to quadratic curvature decay.

	\begin{lem}\label{lem:Rmbound2}
		Suppose $(M,g, \nabla f)$ is a complete cohomogeneity one gradient expanding Ricci soliton over $S^1 \times \mathbb{R}^3$ or $S^2 \times \mathbb{R}^2$, as considered in the setup of this paper. Then, the curvature satisfies the following bound for some $C>0$. Moreover, $C$ can be chosen uniformly in the initial conditions (so $C \equiv C(a_0, f_0)$ or $C \equiv C(b_0, f_0)$, depending on the topology)
		
		\begin{equation}
			|\textnormal{Rm}|_g (r) \displaystyle \leqslant \frac{C}{r^2}
		\end{equation}
		
	\end{lem}
	
	\begin{proof}
		
		\noindent Suppose $(M, g, \nabla f)$ is an expanding soliton as in the hypothesis in the case of either topology. Consider the quantities $A : = a'/a$ and $B: = b'/b$. Then, we can rewrite \cref{eq:aeq} and \cref{eq:beq} as
		
		\begin{equation*}
			A' + A^2 = - 2AB + Af' + 1
		\end{equation*}
		
		\begin{equation*}
			B' + B^2 = -B^2 - AB + Bf' + 1 + \frac{1}{b^2}
		\end{equation*}
		
		\noindent We rewrite the equations as 
		
		\begin{equation*}
			A' + A(A+2B-f')=1
		\end{equation*}
		
		\begin{equation*}
			B' + B(A+2B-f')=1+ \frac{1}{b^2}
		\end{equation*}
		
		\noindent Define $F : [1, \infty) \to \mathbb{R}$ to be the function satisfying $F'(r) = (A+2B)(r) - f'(r)$ and $F(1) = 0$. Note that $F'(r) = O(r)$ by \Cref{thm:fpbound} and \Cref{lem:absimplebound}. Using the definitions of $A$ and $B$ and soliton equation \cref{eq:feq}, we see that $F'' = A' + 2B' - f'' = 1 - (\frac{a'}{a})^2 - 2(\frac{b'}{b})^2$, which is bounded uniformly in the initial conditions on $[1,\infty)$ by the result of \Cref{lem:absimplebound}. Then, we have the equations
		
		\[
		A' + AF' = 1 \quad\quad B' + BF' = 1 + \frac{1}{b^2}
		\]
		
		\noindent which we can solve to get
		
		\begin{equation*}
			A(r) = e^{-F(r)}\int_{1}^{r} e^{F(u)} du + C_Ae^{-F(r)}
		\end{equation*}
		
		\begin{equation*}
			B(r) = e^{-F(r)}\int_{1}^{r} e^{F(u)} du + e^{-F(r)}\int_{1}^{r} e^{F(u)}\frac{1}{b^2(u)}du + C_Be^{-F(r)}
		\end{equation*}
		
		\noindent Note that $C_A$ and $C_B$ vary continuously in the initial conditions, since $F(1)$, $A(1)$ and $B(1)$ vary continuously as well. Then, using the upper bound on $f'(r)$ from \Cref{thm:fpbound} and the bounds on $A$ and $B$ from \Cref{lem:absimplebound}, we see that the last term in both equations decays faster than $e^{-\epsilon r^2/2}$, for some $\epsilon > 0$ uniform in the initial conditions.
		
		\medskip
		
		\noindent \textbf{Claim: } $\displaystyle A(r) = \frac{1}{F'(r)} + O\left(\frac{1}{F'(r)^3}\right)$ and $\displaystyle B(r) = \frac{1}{F'(r)} + O\left(\frac{1}{F'(r)^3}\right)$
		
		\smallskip
		
		\noindent \textit{Proof of Claim:} We first prove the claim for $A$. By the sentence preceding the claim, the term $C_Ae^{-F(r)}$ decays faster than any polynomial in $F'(r)^{-1}$, so it is enough to analyze the integral term. First, consider the quantity
		\[
		\frac{
			\displaystyle \int_0^r e^{F(u)}\, du
		}{
			\displaystyle \frac{e^{F(r)}}{F'(r)}
		}.
		\]
		
		\noindent Applying L’Hôpital’s rule, we get
		\begin{align*}
			L &= \lim_{r \to \infty} \frac{\int_0^r e^{F(u)}\, du}{\frac{e^{F(r)}}{F'(r)}} =  \lim_{r \to \infty} \frac{e^{F(r)}}{e^{F(r)} \left(1 - \frac{F''(r)}{(F'(r))^2}\right)} 
			= \lim_{r \to \infty} \frac{1}{1 - \frac{F''(r)}{(F'(r))^2}} =1.
		\end{align*}
		
		\noindent where the last step follows since $F''$ is bounded and $F'(r) = O(r)$. This implies the asymptotic equivalence
		\[
		e^{-F(r)} \int_0^r e^{F(u)}\, du = \frac{1}{F'(r)} +  o\left( \frac{1}{F'(r)} \right) \quad \text{as } r \to \infty.
		\]
		
		\bigskip 
		
		\noindent Now, we consider the limit
		\[
		L = \limsup_{r \to \infty}
		\frac{
			\displaystyle e^{-F(r)} \int_0^r e^{F(u)}\, du - \frac{1}{F'(r)}
		}{
			\displaystyle \frac{1}{(F'(r))^3}
		}.
		\]
		
		\noindent Multiplying the numerator and denominator by \( e^{F(r)} \), this becomes:
		\[
		L = \limsup_{r \to \infty}
		\frac{
			\displaystyle \int_0^r e^{F(u)}\, du - \frac{e^{F(r)}}{F'(r)}
		}{
			\displaystyle \frac{e^{F(r)}}{(F'(r))^3}
		}.
		\]
		
		\noindent Applying L’Hôpital’s Rule (note that this becomes an inequality for the lim sup), we have 
		\[
		L \leqslant \limsup_{r \to \infty}
		\frac{e^{F(r)} \cdot \frac{F''(r)}{(F'(r))^2}}{e^{F(r)} \left( \frac{1}{(F'(r))^2} - \frac{3 F''(r)}{(F'(r))^4} \right)}
		=
		\limsup_{r \to \infty}
		\frac{F''(r)}{1 - \frac{3 F''(r)}{(F'(r))^2}}.
		\]
		
		\noindent Since \( F'(r) \to \infty \) and \( F''(r) \) is bounded, this is \( O(1) \). Thus,
		\[
		e^{-F(r)} \int_0^r e^{F(u)}\, du = \frac{1}{F'(r)} + O\left( \frac{1}{(F'(r))^3} \right) \quad \text{as } r \to \infty.
		\]
		
		\noindent

		\bigskip
		
		\noindent For $B$, we see that the first term is identical to that of $A$, providing the leading order term $\frac{1}{F'(r)}$. In the second term, define $\tilde{F} := F - 2\text{ln}(b)$. Then, we can rewrite this term as
		
		\begin{equation*}
			e^{-F(r)}\int_{1}^{r} e^{F(u)}\frac{1}{b^2(u)}du  = \frac{1}{b^2(r)} \left( e^{-\tilde{F}(r)} \int_{1}^{r} e^{\tilde{F}(u)}du \right)
		\end{equation*}
		
		\noindent Now, the term in parentheses above can be shown by a similar argument to decay with leading term $ \displaystyle \frac{1}{\tilde{F}'(r)}$ (Note that $\tilde{F}''  = 1 - (\frac{a'}{a})^2 - 2\frac{b''}{b}$. This is bounded uniformly in the initial conditions on $[1,\infty)$ by Lemmas\autoref{lem:absimplebound} and\autoref{lem:Rmbound1} since $-b''/b$ is a sectional curvature, allowing the argument for the first term for $B$ to be applied to the second term as well). Thus, the second term for $B$ is $\displaystyle O\left(\frac{1}{b(r)^2\tilde{F}'(r)}\right)$. Since $b(r)$ is known to grow at least linearly by \Cref{lem:ablinbelow}, we can combine the decay rate of all $3$ terms for $B$ to see that the leading term is $1/F'(r)$ and that all other terms decay at least as fast as $r^{-3}$, proving the claim for $B$ as well. $\blacksquare$
		
		\medskip
		
		\noindent Using the claim and the fact that $F'(r)$ is of linear growth (uniform in the initial conditions), we see that
		
		\begin{equation*}
			\frac{a''}{a}(r) = A'(r) + A(r)^2 = 1 - A(r)F'(r) + A(r)^2 = O\left(\frac{1}{F'(r)^2}\right) + O\left(\frac{1}{r^2}\right)= O\left(\frac{1}{r^2}\right)
		\end{equation*}
		
		\begin{equation*}
			\frac{b''}{b}(r) = B'(r) + B(r)^2 = 1 +\frac{1}{b^2(r)} - B(r)F'(r) + B(r)^2 = O\left(\frac{1}{F'(r)^2}\right) + O\left(\frac{1}{r^2}\right) = O\left(\frac{1}{r^2}\right)
		\end{equation*}
		
		\begin{equation*}
			\frac{a'b'}{ab}(r) = AB = O\left(\frac{1}{F'(r)^2}\right) = O\left(\frac{1}{r^2}\right)
		\end{equation*}
		
		\begin{equation*}
			\frac{1-(b')^2}{b^2}(r) = \frac{1}{b^2} - B^2 = O\left(\frac{1}{r^2}\right) + O\left(\frac{1}{F'(r)^2}\right) = O\left(\frac{1}{r^2}\right)
		\end{equation*}
		
		\noindent Thus, we have shown that all sectional curvatures (refer Appendix A) decay as $r^{-2}$. 
	\end{proof}

	\noindent Now that we know that the curvature tensor decays as $r^{-2}$, we can make a more precise statement about the asymptotics of $f$.
	
	\begin{lem}\label{lem:fpplusr}
		Suppose $(a,b,f)$ satisfy equations \cref{eq:feq,eq:aeq,eq:beq}, with boundary conditions either \cref{S1R3a,S1R3b} or \cref{S2R2a,S2R2b}. Then, the quantity $f'(r)+r$ has a finite limit at infinity, denoted $K$. In addition, we have the following inequality for $r>0$:

		\begin{center}
			$\displaystyle 0 < |f'(r)+r-K| < \frac{C}{r}$
		\end{center}
		
		\noindent  Moreover, $C \equiv C(a_0, f_0)$ or $C \equiv C(b_0, f_0)$ can be chosen uniformly in the initial conditions and $K$ is continuous in the initial conditions.
	\end{lem}
	
	\begin{proof}
		Equation \cref{eq:feq} can be written as
		
		\begin{center}
			$\displaystyle f''+1=\frac{a''}{a}+2\frac{b''}{b}$
		\end{center}
		
		\noindent As $-\frac{a''}{a}$ and $-\frac{b''}{b}$ are components of the curvature tensor (components $\text{Rm}_{1221}$ and $\text{Rm}_{1331}$ respectively; refer Appendix A for details), they decay as $r^{-2}$. Thus, for some constant $C>0$ (which can be chosen uniformly in the initial condtions) by \Cref{lem:Rmbound2}, for $r>0$, we have
		
		\begin{center}
			$\displaystyle 0 \leqslant |f''+1| \leqslant \frac{C}{r^2}$
		\end{center}
		
		\noindent As $f'(r)+r$ has derivative equal to $f''(r)+1$, and since $f''(r)+1$ decays like $r^{-2}$, we see that $f'(r)+r$ has a finite limit as $r \to \infty$, denoted by $K$. The uniformity of $C$ allows us to see that $K$ is continuous in the initial conditions.
		
		\medskip
		
		\noindent Fix $r,s > 0$ and integrate the inequality above to get
		
		\begin{center}
			$\displaystyle 0 \leqslant \left| \int_{r}^{s}(f''(t)+1) dt \right| \leqslant \int_{r}^{s}|f''(t)+1|dt \leqslant \frac{C}{r}-\frac{C}{s}$
		\end{center}
		
		\noindent We thus see that 
		
		\begin{center}
			$\displaystyle 0 \leqslant |f'(r)+r-f'(s)-s| \leqslant \frac{C}{r}-\frac{C}{s}$
		\end{center}
		
		\noindent  Taking the limit in the above inequality as $s \to \infty$ gives the result. 
	\end{proof}
	
	\noindent Now, we can understand the asymptotic behavior of $a$ and $b$. 
	
	\begin{lem}\label{lem:coneterms}
		Suppose $a,b,f$ satisfy equations \cref{eq:feq,eq:aeq,eq:beq} with initial conditions either \cref{S1R3a,S1R3b} or \cref{S2R2a,S2R2b}, depending on the topology. Then we have the following inequalities for $r>1$, which can be chosen uniformly in the initial conditions:
		
		\begin{center}
			$\displaystyle 0 \leqslant \left|1+\frac{a'f'}{a} \right| \leqslant \frac{C}{r^2}$
			
			\medskip
			
			$\displaystyle 0 \leqslant \left|1+\frac{b'f'}{b}\right| \leqslant \frac{C}{r^2}$
		\end{center}
		
		\noindent where $C \equiv C(a_0, f_0)$ or $C \equiv C(b_0, f_0)$ can be chosen uniformly in the initial conditions.
	\end{lem}
	
	\begin{proof}
		From \cref{eq:aeq}, we see that
		
		\begin{center}
			$\displaystyle 1+\frac{a'f'}{a}=\frac{a''}{a}+\frac{2a'b'}{ab}$
		\end{center}
		
		\noindent The terms on the right hand side comprise a component of the Ricci curvature (refer Appendix A for details), so the RHS must decay as $r^{-2}$, proving the first half of the lemma. The result for $b$ follows analogously using \cref{eq:beq}.
	\end{proof}
	
	\noindent Now, we will show that $a$ and $b$ are bounded above and below by linear functions.
	
	\begin{lem}\label{lem:abbounds}
		Consider \cref{eq:feq,eq:aeq,eq:beq}, with initial conditions \cref{S1R3a,S1R3b} (in the $S^1 \times \mathbb{R}^3$ case) or \cref{S2R2a,S2R2b} (in the $S^2 \times \mathbb{R}^2$ case). Then, there exists constants $\alpha_1, \alpha_2, \beta_1, \beta_2>0$ and $C,C'D,D' \in \mathbb{R}$ so that we have $\alpha_1 (r+C) < a(r) < \alpha_2 (r+C')$ and $\beta_1 (r+D') < b(r) < \beta_2 (r+D')$. Additionally, $\alpha_1, \alpha_2, \beta_1, \beta_2$ depend continuously on the initial conditions.
	\end{lem}
	
	\begin{proof}
		We will carry out the proof for $a$ in the $S^1 \times \mathbb{R}^3$ case; similar arguments hold for $b$ (the $\frac{1}{b}$ term can be bounded above by a constant $C_b$, since $b$ is increasing; consider the quantity $b+C_b$). Note that \Cref{lem:ablinbelow} provides lower bounds; we show the upper bound. Suppose we have an initial condition inside a compact set of the form $(a_0, f_0) \in F \subset \mathbb{R}^+ \times \mathbb{R}^-$.
		
		\medskip
		
		\noindent \textbf{Step 1:} Using Lemmas\autoref{lem:abmon}, \autoref{lem:fmon}, and\autoref{lem:fpplusr} in \cref{eq:aeq}, for a large $r_0>0$, we can write the following inequality on the interval $[r_0,\infty)$:
		
		\begin{equation*}
			a''  \leqslant a'(-r+C)+a
		\end{equation*}
		
		\medskip
		
		\noindent  where $C$ depends on the constant $K$ from \Cref{lem:fpplusr}, and $a$ has initial conditions $a(r_0)$ and $a'(r_0)$. 
		
		\medskip 
		
		\noindent \textbf{Step 2:} We notice that the values $a(r_0)$ and $a'(r_0)$ are bounded from above and below (with nonzero lower bounds) by our choice of compact set $F$. Additionally, $C$ can be chosen uniformly over $F$ as well. Then, by Part 1 of \Cref{lem:ODE}, we have a bound $a(r) \leqslant C_1(r-r_0) + a(1)$ for some $C_1 >0$ which can be chosen uniformly over $F$ and the constant $C$. Thus, all $a$ with initial conditions in $K$ are bounded from below on $[r_0, \infty)$ by a linear function. The proof for $b$ is similar, following from equation \cref{eq:beq}.
	\end{proof}

	\noindent Now, we can show that $a$ and $b$ are asymptotically linear.
	
	\begin{lem}\label{lem:limcont}
		\begin{enumerate}
			\item With $K$ as in \Cref{lem:fpplusr}, the quantities $\frac{a(r)}{r-K}$ and $\frac{b(r)}{r-K}$ tend to finite limits as $r \to \infty$. Moreover, the quantities $a'(r)$ and $b'(r)$ tend to the same limits, respectively denoted $a'_\infty$ and $b'_\infty$.
			\item The quantities $a'_\infty$ and $b'_\infty$ are continuous in the initial conditions.
		\end{enumerate}
	\end{lem}
	
	\begin{proof}
		First, we prove the first statement. We will carry out the proof for $a$; the proof for $b$ is analogous. We consider the quantity $\displaystyle \left|\frac{a(r)}{r-K}-a'(r)\right|$ for $r \geqslant \text{max}(0,K)$
		
		\begin{align*}
			\displaystyle 	\left|\frac{a(r)}{r-K}-a'(r)\right|&=\left|\frac{a(r)}{r-K}+\frac{a(r)}{f'(r)}-\frac{a(r)}{f'(r)}-a'(r)\right| \\
			&\leqslant \left|\frac{a(r)}{r-K}+\frac{a(r)}{f'(r)}\right|+\left|\frac{a(r)}{f'(r)}+a'(r)\right|\\
			&= a(r)\left|\frac{f'(r)+r-K}{f'(r)(r-K)}\right|+\left|\frac{a(r)+a'(r)f'(r)}{f'(r)}\right|\\
		\end{align*}
		\noindent By Lemmas\autoref{lem:fpplusr}, \autoref{lem:coneterms}, and\autoref{lem:abbounds} we see that the first term on the RHS is bounded by $\frac{C}{r^2}$ for some constant $C>0$, while the second term is also bounded by $\frac{C}{r^2}$, both for $r>\text{max}(0,K)+1$.
		
		\medskip
		
		\noindent Thus, we have for $r>\text{max}(0,K)+1$
		
		\begin{equation}
			\left| \frac{a(r)}{r-K}-a'(r) \right| < \frac{C}{r^2}.
		\end{equation}
		
		\noindent This can be rewritten as
		
		\begin{equation}\label{eq:apestimate}
			\displaystyle \left|(r-K) \frac{d}{dr}\left(\frac{a(r)}{r-K} \right) \right| < \frac{C}{r^2}
		\end{equation}
		
		\noindent This shows that the derivative of $\frac{a(r)}{r-K}$ decays as $r^{-3}$. Thus, we see that $\frac{a(r)}{r-K}$ and thus $\frac{a(r)}{r}$ tends to a finite limit as $r \to \infty$. By \Cref{lem:ablinbelow}, this limit is at least $c>0$, so it must be positive. 
		
		\medskip 
		
		\noindent Now, we see that $a'$ and $b'$ reach finite limits as $r \to \infty$. We denote these limits by $a'_\infty$ and $b'_\infty$. This concludes the proof of (1).
		
		\bigskip \noindent Now, for (2), observe that the constant $C$ in equation \Cref{eq:apestimate} depends uniformly on the initial conditions, as it only depends on similarly behaved constants from Lemmas\autoref{lem:fpplusr}, \autoref{lem:coneterms}, and\autoref{lem:abbounds}. Thus, the function $\frac{a(r)}{r-K}$ depends continuously on the initial conditions and its derivative is bounded by $C/r^3$ for a uniform constant $C$ (on an interval of the form $[r_0, \infty)$, by the continuity of $K$). Thus, $\lim_{r \to \infty} \frac{a(r)}{r-K}$ is continuous in the initial conditions. Since we know from (1) that this constant is equal to $a'_\infty$, we see that the slope $a'_\infty$ is continuous in the initial conditions. The analogous argument shows the continuity of $b'_\infty$ as well.
	\end{proof}

	\noindent Geometrically, this suggests that the gradient expanding solitons over $S^1 \times \mathbb{R}^3$ or $S^2 \times \mathbb{R}^2$ are asymptotically conical. We formalize this as follows:
	
	\medskip
	
	\noindent 
	
	\begin{thm}\label{thm:GHC}
		In either of the $S^1 \times \mathbb{R}^3$ and $S^2 \times \mathbb{R}^2$ cases, fix a cohomogeneity one gradient expanding soliton $(M, g, \nabla f)$ and a point $p$ with $r(p)=0$. Consider any sequence $\lambda_i \to 0$. Then, we have Gromov-Hausdorff convergence of  $(M,\lambda^2_i g, p)$ to a cone over the link $S^2 \times S^1$.
	\end{thm}
	
	\begin{proof}
		Consider the sequence $\nu_i := \frac{1}{\lambda_i}$. Then, the rescaled metric $g_i := \lambda^2_i g$ is a warped product given by
		
		\begin{equation*}
			g_i = dr^2 + a_i(r)^2g_{S^1} + b_i(r)^2g_{S^2}
		\end{equation*}
		
		\noindent where $a_i(r) =  \frac{a(\nu_i r)}{\nu_i}$ and $b_i(r) =  \frac{b(\nu_i r)}{\nu_i}$. By (6.4), since we know that since $\lim_{r \to \infty} \frac{a(r)}{r-K} = a'_\infty$ and that the derivative of $\frac{a(r)}{r-K}$ decays as $r^{-3}$, we have the inequality for $r$ greater than some large $r_0$ which is greater than $K$:
		
		\begin{center}
			$\left| \dfrac{a(r)}{r-K}-a'_\infty \right| \leqslant \dfrac{C}{r^2}$
		\end{center}
		
		\noindent We also have
		
		\begin{center}
			$\left| \dfrac{a(r)}{r}-a'_\infty \right| \leqslant \left|\dfrac{a(r)}{r} - \dfrac{a(r)}{r-K}\right|  + \left| \dfrac{a(r)}{r-K}-a'_\infty \right| \leqslant \dfrac{C}{r} $
		\end{center}
		
		\noindent by using \Cref{lem:abbounds} on the first term. Thus, we have 
		
		\begin{center}
			$|a(r)-a'_\infty r| \leqslant C$
		\end{center}
		
		\noindent for $r>r_0$. We can immediately extend this bound to the interval $[0,\infty)$ (for a different $C$) by compactness of $[0,r_0]$. From this, we have
		
		\begin{center}
			$\left| \dfrac{a(\nu_i r)}{\nu_i}-a'_\infty r \right| \leqslant \dfrac{C}{\nu_i}$
		\end{center}
		
		\noindent which implies that the functions $a_i(r)$ converge uniformly to $a'_\infty r$. A similar result holds for $b_i$. From this, it is easy to see that since the metrics on the rescaled solitons converge uniformly to the metric of the asymptotic cone, we have the required Gromov-Hausdorff convergence.
	\end{proof}

	\noindent Based on the results of this discussion, we make the following definition.
	
	\begin{defn}\label{def:cone}
		A cohomogeneity one gradient expanding soliton $(M,g, \nabla f)$ as considered in this paper is called \textbf{asymptotically conical} if it satisfies the following conditions:
		
		\begin{enumerate}
			\item $|\textnormal{Rm}|_g(r) \leqslant \frac{C}{r^2}$ for some $C>0$.
			\item $|\nabla f|(r) \leqslant r+C$ for some $C \in \mathbb{R}$.
			\item There exists a sequence $\lambda_i \to 0$ so that $(M, \lambda^2_i g, p)$ Gromov-Hausdorff converges to a cone metric $\gamma$ with link $(S^2 \times S^1, h)$, where the metric $h$ admits an isometric action of $\textnormal{SO}(3) \times \textnormal{SO}(2)$.
		\end{enumerate}
	\end{defn}
	
	\noindent From Lemmas\autoref{lem:Rmbound2}, \autoref{lem:fpplusr} and \Cref{thm:GHC}, we see that in the case of each topology, the solitons in the $2$-parameter family from \Cref{thm:2pf} are asymptotically conical as in \Cref{def:cone}. We note that these solitons were shown to be asymptotically conical in previous work of \cite{NW24}, \cite{Win21}, and \cite{BDGW15}. We have provided an alternate proof (using our slightly different definition of asymptotically conical) both for the sake of completeness of our work and additionally to use certain estimates from this section in future sections.

	\section{Relating Expanding Solitons to Asymptotic Cones}
	
	\noindent As we know that our solitons $(M,g)$ are asymptotic to cones $(\mathbb{R}^+ \times S^2 \times S^1, \gamma)$, it will be important to understand how close $g$ is to $\gamma$ as well as the value of the distance from the singular orbit of $M$ at which this happens. In this section, we show that the assignment of the asymptotic cone to a soliton is a continuous map by using results from the previous section, and also provide the setup to show that this assignment is a proper map. To do this, we introduce a notion called uniform $\epsilon$-conicality and show that the closeness of an expanding soliton considered in this paper to its asymptotic cone is determined by the geometry of the cone. This is the main result of this section, used to prove the properness result of Section 9.
	
	\medskip
	
	\noindent First, we make use of the following map, constructed in \cite{BC23} which allows us to embed into a soliton its asymptotic cone.
	
	\begin{lem}\label{lem:embedcone}
		Consider a cohomogeneity one gradient expanding Ricci soliton $(M,g,\nabla f)$ asymptotic to a cone $\gamma = ds^2 + s^2h$ over a link $(S^2 \times S^1, h)$ as in \Cref{def:cone}. Then, there exists a smooth embedding $\iota: \mathbb{R}_+ \times S^2 \times S^1 \to M$ satisfying the following
		
		\begin{enumerate}
			\item $M - \iota((s, \infty) \times S^2 \times S^1)$ is compact for all $s \geqslant 0$.
			\item $\iota^*(\nabla f)  = -s\partial_s$, where $s$ is the coordinate on $\mathbb{R}_+$ and $\partial_s$ is the corresponding vector field.
			\item The pullback metric $\iota^*g$ is smooth
		\end{enumerate}
		
		\noindent Further, suppose that we have the following bounds, where the constants $\alpha$ and $A_i$ are positive, and $m \geqslant 0$
		
		\begin{equation}\label{eq:conebounds}
			\displaystyle \textnormal{inj}_\gamma \geqslant \alpha s \hspace{3cm} |\nabla^{m,\gamma} \textnormal{Rm}_\gamma| \leqslant \dfrac{A_m}{s^{2+m}}
		\end{equation}
		
		\noindent Then, there exists $S_0 \equiv S_0(\alpha,A_0)$ such that we have the following quantitative asymptotics of $\iota^*g$ to $\gamma$ on $(S_0, \infty) \times S^2 \times S^1$
		
		\begin{equation}\label{eq:conedecay}
			\displaystyle |\nabla^{m,\gamma}(\iota^*g-\gamma)| \leqslant \dfrac{C_m(\alpha, A_0, \dots, A_m)}{s^{2+m}}
		\end{equation}
		
		\noindent The map $\iota$ is unique in the following sense: suppose that $\iota'$ is another such smooth map satisfying conditions (1)-(3). Then, we must have $\iota' = \iota \circ (\textnormal{Id}_\mathbb{R^+}, \psi)$, where $\psi:(S^2 \times S^1, h) \to (S^2 \times S^1, h)$ is an isometry.
	\end{lem}
	
	\noindent The conclusion of the lemma is essentially unchanged from Lemma 2.9 of \cite{BC23}. The hypotheses that the solitons in the cohomogeneity one $2$-parameter families from \Cref{thm:2pf} are asymptotically conical was verified in the discussion following \Cref{def:cone}.
	
	\medskip
	
	\noindent In our warped product setting, it is clear that $\alpha = \text{inj}_{S^2 \times S^1}h = \pi \text{min}\{a'_\infty, b'_\infty\}$. Since $(M,g)$ is cohomogeneity one, we know that for any $p \in M$, that $\nabla f(p) \equiv f'(r) {\partial_r}|_{p}$ depends only on the coordinate $r$ and not on the coordinates of $S^2 \times S^1$. By Part 2 of \Cref{lem:embedcone}, since the trajectories of $\iota$ are the integral curves of this vector field, the coordinates on $S^2 \times S^1$ are unchanged along $\iota$.
	
	\medskip
	
	\noindent Thus, we may write $\iota(s,z) = (d(s),z)$, for some smooth function $d: \mathbb{R}^+ \to \mathbb{R^+}$ (here, we precompose $\iota$ with the required isometry of $S^2 \times S^1$ if necessary to ensure that the map $\iota$ leaves the coordinate on every $S^2 \times S^1$ orbit unchanged).
	
	\medskip
	
	\noindent This lemma allows us to quantify the closeness between the soliton metric and the asymptotic cone via the following definition.
	
	\begin{defn}\label{def:epsilonconical}
		Fix $\epsilon >0$ and consider a cohomogeneity one gradient expanding soliton $(M,g,\nabla f)$ with topology $S^1 \times \mathbb{R}^3$ or $S^2 \times \mathbb{R}^2$ asymptotic to the cone $(\mathbb{R}^+ \times S^2 \times S^1, \gamma = ds^2 + s^2h)$. Suppose that there exists $S_0 >0$ such that on $[S_0,\infty) \times S^2 \times S^1$, we have the following bound for all $m \leqslant 10$:
		
		\begin{equation}\label{eq:cone_soliton}
			\displaystyle |\nabla^{\gamma,m}(\iota^*g-\gamma)| \leqslant \epsilon
		\end{equation}
		
		\noindent Then, $(M,g, \nabla f)$ is said to be $\epsilon$-conical at distance $d(S_0)$ from the tip $r=0$, where $d: \mathbb{R^+} \to \mathbb{R^+}$ is the map defined as above with $\iota(s,z) = (d(s),z)$.
	\end{defn}
	
	\noindent Using \Cref{def:epsilonconical}, we will show that curvature and injectivity radius bounds on the links of the asymptotic cone $\gamma$ of a soliton $(M, g)$ are sufficient to control the distance from the singular orbit at which the solitons are $\epsilon$-conical.
	
	\begin{lem}\label{lem:symcone}
		Suppose $(M,g_i, \nabla f_i)$ is a collection of cohomogeneity one gradient expanding Ricci solitons asymptotic to cone metrics $\gamma_i$ with links $(S^2 \times S^1, h_i)$, where the link metrics are $h_i = (a'_{\infty,i})^2g_{S^1}+(b'_{\infty,i})^2g_{S^2}$. Suppose we have the following bound for all $i$:
		
		\[
		\min \{ a'_{\infty,i}, b'_{\infty,i} \} \geqslant c
		\]
		
		\noindent Then, for any fixed $\epsilon >0$, there is a fixed constant $S_0 \equiv S_0(\epsilon, c) >0$ so that each $(M,g_i, \nabla f_i)$ is $\epsilon$-conical at $d_i(S_0)$. Additionally, we have the following bounds:
		\begin{align}
			a'_{\infty,i}(1-\epsilon)^{1/2}S_0 &\leqslant a_i(d_i(S_0)) \leqslant a'_{\infty,i}(1+\epsilon)^{1/2}S_0 \label{eq:a_estimate} \\
			b'_{\infty,i}(1-\epsilon)^{1/2}S_0 &\leqslant b_i(d_i(S_0)) \leqslant b'_{\infty,i}(1+\epsilon)^{1/2}S_0 \label{eq:b_estimate}
		\end{align}
		
	\end{lem}
	
	\begin{proof}
		\noindent For a given expander $(M, g, \nabla f)$ asymptotic to the cone $\gamma$ with asymptotic slopes $a'_\infty, b'_\infty$, the nonzero sectional curvatures (calculated using the coordinate system in Appendix A) of $\gamma$ are 
		
		\begin{center}
			$\displaystyle \text{Rm}_{\gamma,2332} = \text{Rm}_{\gamma, 2442} = -\frac{1}{s^2}  \hskip 1cm \text{Rm}_{\gamma, 3443} = \frac{1-(b'_\infty)^2}{b'^2_\infty}\frac{1}{s^2}$ 
		\end{center}
		
		\noindent The bound in the hypothesis implies that $\text{inj}(h_i)$ is bounded below, or that $\text{inj}_{\gamma_i}(s) \geqslant \alpha s$ for some constant $\alpha > 0$. Additionally, since $b'_{\infty,i}$ is bounded from below, we can differentiate the curvature terms to see that we have the required bounds on $|\nabla^{m,\gamma} \textnormal{Rm}_\gamma|$ in the second part of \cref{eq:conebounds}. Thus, by \Cref{lem:embedcone}, there exists $R_0 \equiv R_0(c)$ such that the following bound holds for all $i$ and $m \leqslant 10$ on $(R_0, \infty) \times S^2 \times S^1$:
		
		\begin{equation*}
			\displaystyle |\nabla^{m,\gamma_i}(\iota_i^*g_i-\gamma_i)| \leqslant \dfrac{C_m(\alpha, A_0, \dots, A_m)}{s^{2+m}}
		\end{equation*}
		
		\noindent Now, choose $S_0 > R_0$ so that the bound on the right hand side (for each $m \leqslant 10$) is less than $\epsilon$ for $s \geqslant S_0$. This proves the first statement of the lemma.
		
		\medskip
		
		\noindent For the second statement, for $m=0$, we have the following inequality on $(S_0, \infty) \times S^2 \times S^1$ for all $i$:
		\begin{equation*}
			(1-\epsilon) \gamma_i \leqslant \iota_i^* g_i \leqslant (1+\epsilon)\gamma_i
		\end{equation*}
		
		\noindent By plugging in the appropriate unit vectors on $(\mathbb{R} \times S^1 \times S^2, \gamma_i)$ based at $S_0$ tangent to the $S^1$ and $S^2$ directions into the above inequalities, we obtain inequalities \Cref{eq:a_estimate} and \Cref{eq:b_estimate}, respectively.
	\end{proof}

	\noindent With an additional diameter assumption on the links $h_i$, we will further show that the sequence $d_i(S_0)$ is bounded. In the remainder of this section, we will show this by contradiction; more specifically, assuming that \{$d_i(S_0)\}$ is unbounded, we take a geometric limit to produce a certain expanding soliton with two ends and identify a contradiction. For this, we will need some technical lemmas that allow us to take this limit and to show that two-ended expanding solitons satisfying certain conditions do not exist. The first lemma below shows that curvature bounds on a soliton provide lower bounds on the size of its $S^2$-orbit. 
	
	\begin{lem}\label{lem:bbound}
		Suppose $(M,g)$ is a cohomogeneity one gradient expanding Ricci soliton as considered in this paper. If $(M,g)$ satisfies the bound $|\text{Rm}|_g \leqslant C$ for some constant $C>0$ on a ball $B_g(p,D)$ (with $D>1$) where $r(p)=r^*$, then we have the following bound:
		
		\[
		b(r^*) \geqslant \min\left\{ \frac{1}{2}, \frac{1}{2C^{1/2}} \right\}
		\]
	\end{lem}
	
	\begin{proof}
		\noindent
		The curvature bound on the ball provides the following inequality on Rm$_{3443}$ (refer Appendix A for the calculation of the sectional curvatures) 
		\[
		|\text{Rm}_{3443}| = \left| \frac{1 - (b')^2}{b^2} \right| \leqslant C
		\]
		From this, we deduce the lower bound in the statement in the lemma. If the bound does not hold, the inequality \( (b')^2 \geqslant (1 - C b^2) \) implies that \( b'(r^*) \geqslant \sqrt{3}/2 \). Using the monotonicity of \( b \) from \Cref{lem:abmon}, \( b'(r) \geqslant \sqrt{3}/2 \) for all \( r \leqslant r^* \), but this implies that \( b \) becomes negative (in finite distance from \( r^* \)) inside \( B_g(p, D) \).
	\end{proof}
	
	\noindent Next, we will show that certain kinds of cohomogeneity one gradient expanding solitons with $2$ ends cannot arise as limits of solitons in the $2$-parameter families we consider in this paper. First, we consider certain classes of cohomogeneity one Einstein metrics.
	
	\begin{lem}\label{lem:einstein}
		Suppose $(M,g=dr^2+a(r)^2g_{S^1}+b(r)^2g_{S^2})$ is a cohomogeneity one Einstein manifold satisfying $\text{\normalfont Ric}_g + g = 0$ on $S^1 \times \mathbb{R}^3$ or $S^2 \times \mathbb{R}^2$. Then, we have the equalities $b'=Ca$ and $b''=Ca'$ for some constant $C>0$.
	\end{lem}
	
	\begin{proof}
		Setting $f \equiv 0$ in \cref{eq:feq,eq:aeq,eq:beq}, we see that the Einstein equations are
		
		\begin{equation*}
			\displaystyle \frac{a''}{a} = - \frac{2a'b'}{ab}+1
		\end{equation*}
		
		\begin{equation*}
			\displaystyle \frac{b''}{b} =   \frac{1-(b')^2}{b^2}-\frac{a'b'}{ab}+1
		\end{equation*}
		
		\begin{equation*}
			\displaystyle \frac{a''}{a} +  \frac{2b''}{b}-1=0
		\end{equation*}
		
		\noindent From the first and third equations, we see that
		
		\begin{center}
			$\displaystyle \frac{b''}{b} = \frac{a'b'}{ab}$
		\end{center}
		
		\noindent Then, the function $\frac{b'}{a}$ satisfies
		
		\begin{center}
			$\displaystyle \left(\frac{b'}{a}\right)' = \frac{b''}{a}-\frac{a'b'}{a^2} = \frac{b}{a}\left(\frac{b''}{b}-\frac{a'b'}{ab}\right)=0$
		\end{center}
		
		\noindent Thus, $\frac{b'}{a}$ must be a constant, implying that conclusion of the lemma.
	\end{proof}

	\noindent Now, we rule out certain Einstein metrics on $\mathbb{R} \times S^2 \times S^1$. The key intuition underlying this proof is that the second equation (analogous to \cref{eq:aeq}) for Einstein metrics cannot be satisfied near $-\infty$ since the curvature of the $S^2$ would become too large.
	
	\begin{lem}\label{lem:noeinstein}
		There are no cohomogeneity one Einstein metrics of the form $\text{\normalfont Ric}_g+g=0$ on the space $\mathbb{R} \times S^2 \times S^1$ arising as Cheeger-Gromov limits of doubly warped product expanding solitons $(M,g_i = dr^2 + a_i(r)^2g_{S^1}+b_i(r)^2g_{S^2})$, where $a_i$ and $b_i$ are monotonically increasing functions.
	\end{lem}
	
	\begin{proof}
		Suppose that such an Einstein metric exists. Then, the Einstein equations would have a solution $(a,b)$ on the interval $(-\infty, \infty)$. As we are assuming that such an Einstein metric is a limit of doubly warped product expanding solitons, the monotonicity properties for such solitons carry over to give us the inequalities $a',b' \geqslant 0$.
		
		\medskip
		
		\noindent Thus, we see that both $a$ and $b$ approach finite limits as $r$ tends to $-\infty$ by monotonicity.  From \Cref{lem:einstein}, we see that $b'=Ca$, so $b'$ approaches a limit as $r$ tends to $-\infty$. As $b$ itself approaches a finite limit, $b'$ must tend to $0$ as $r$ tends to $-\infty$. Then, we can rewrite the second Einstein equation as 
		
		\begin{equation}\label{eq:beinstein}
			\displaystyle 1+ \frac{1-(b')^2}{b^2} = 2C\frac{a'}{b}
		\end{equation}
		
		\noindent which implies that
		
		\begin{equation}\label{eq:beinstein2}
			\displaystyle b^2+ {1-(b')^2} = 2C{a'b}
		\end{equation}
		
		\noindent From \Cref{eq:beinstein}, we see that $a'$ approaches a limit as $r$ tends to $-\infty$, and this limit must be $0$ as $a \geqslant 0$ everywhere. This leads to a contradiction, as the right hand side of \Cref{eq:beinstein2} approaches $0$ as $r$ tends to $-\infty$, while the left hand side approaches a nonzero value. Thus, there are no such Einstein metrics with two ends arising as such limits.
	\end{proof}
	
	\noindent Now, we show that two-ended gradient expanding solitons satisfying the monotonicity properties (established for one-ended solitons in Section 4) cannot exist. Note that we do not rule out two-ended expanding solitons in general; in fact, in \cite{Ram12}, the existence of a $3$-dimensional gradient expanding soliton on $\mathbb{R} \times S^1 \times S^1$ is established. In our situation, the topology and monotonicity properties will be key to the proof.
	
	\begin{lem}\label{lem:no2ends}
		There does not exist a cohomogeneity one gradient expanding Ricci soliton on $\mathbb{R} \times S^2 \times S^1$ with the monotonicity properties $a', b' \geqslant 0$ and $f', f'' \leqslant 0$.
	\end{lem}
	
	\begin{proof}

		\noindent Suppose $(M, g, \nabla f)$ is a cohomogeneity one gradient expanding Ricci soliton on $\mathbb{R} \times S^2 \times S^1$ with the given monotonicity properties. Recalling soliton identities \cref{eq:trace} and \cref{eq:bianchi}, we have the following identities for some constant \( C \in \mathbb{R} \):
		\[
		\begin{aligned}
			R + \Delta f + 4 &= 0, \\
			R + |\nabla f|^2 + 2f &= C.
		\end{aligned}
		\]
		\noindent Combining these identities, we see that
		\[
		f = \frac{1}{2}(\Delta f - |\nabla f|^2 + C).
		\]
		\noindent Using the monotonicity properties, we have \( \Delta f = f'' + f'\left(\frac{a'}{a} + \frac{2b'}{b}\right) \leqslant 0 \), so \( f \) is bounded from above.
		
		\medskip
		
		\noindent Now, we make the following claim:
		
		\medskip
		
		\noindent \textbf{Claim:} $(M,g)$ must have bounded sectional curvature on the end where $r \to -\infty$. 
		
		\medskip
		
		\noindent \textit{Proof of Claim:} Suppose this is not the case; then we can find a sequence of points $p_i \in (M,g)$ where $r(p_i) = r_i \to -\infty$ and $|\text{Rm}|_{g}(p_i) \to \infty$. For any $D_i>0$, using \Cref{lem:pp}, we can find a sequence of points $q_i \in B_{g}(p_i,{2D_i}/\sqrt{|\text{Rm}|_{g}(p_i)})$ where 
		
		\begin{center}
			$\displaystyle r(q_i) \leqslant r(p_i)+\frac{d(p_i,q_i)}{\sqrt{|\text{Rm}|_{g}(p_i)}} \leqslant r_i + \frac{2D_i}{\sqrt{|\text{Rm}|_{g}(p_i)}}$
		\end{center}
		
		\noindent along with the bounds $|\text{Rm}|_{g}(q_i) := Q_i \to \infty$ and $|\text{Rm}_{g}| \leqslant 4Q_i$ on $B_{g}(q_i, D_i/\sqrt{Q_i})$. Now choose a sequence $D_i \to \infty$ such that $r(q_i) \to -\infty$. Rescale to get $\tilde{g}_i = Q_i g$, where the new inequality is $|\text{Rm}_{\tilde{g_i}}| \leqslant 4$ on $B_{\tilde{g_i}}(q_i, D_i)$. Then, for any $D>0$, we have $D < D_i$ for large $i$, so $|\text{Rm}|_{\tilde{g_i}} \leqslant C(D)$ on $B_{\tilde{g}_i}(q_i, D)$ for a constant $C(D)$ depending on $D$.
		
		\medskip
		
		\noindent By the monotonicity of $f'$, since $r(q_i)$ is bounded above, we see that $|\nabla f|(q_i)$ is bounded. By Shi's estimates for the associated Ricci flows, we can derive bounds on the derivatives of the curvature on $B_{\tilde{g}_i}(q_i, D)$ as well as higher derivatives of $f$ using the soliton equation. By \Cref{lem:bbound}, we have a uniform lower bound on $b(r(q_i))$. By multiplying $a$ by a constant factor (note that this does not affect the soliton equations), we can assume that $a(r(q_i))$ is bounded from above and below.

		\medskip

		\noindent Thus, by the previous paragraphs, we have the required volume and curvature bounds to take (up to a subsequence) a Cheeger-Gromov limit of $(M, g_i, \nabla \tilde{f}_i, q_i)$ to get a warped product metric $(M_\infty, g_\infty, \nabla f_\infty, q_\infty)$, which satisfies the steady soliton equation. If the size of the $S^2$ orbit remains bounded, $M_\infty$ has topology $\mathbb{R} \times S^2 \times S^1$ and thus has two ends. If it becomes unbounded, then $M_\infty$ has topology $\mathbb{R} \times \mathbb{R}^2 \times S^1$, but $M_\infty/\mathbb{Z}^2 \equiv \mathbb{R} \times \mathbb{S}^1 \times \mathbb{S}^1 \times \mathbb{S}^1$ has two ends.
		
		\medskip
		
		\noindent Thus, we have a steady soliton with two ends, which must split as the product of $\mathbb{R}$ and a compact $3$-dimensional Ricci-flat manifold $N$ by \cite{MW11}. As $N$ must be flat, we see that $(M_\infty, g_\infty)$ is isometric to a quotient of Euclidean space, but this contradicts $|\text{Rm}_{g_\infty}(q_\infty)| \neq 0$. Thus, the claim must be true. \qed
		
		\medskip
		
		\noindent Now, take a sequence of points $p_i$ with $r(p_i) = r_i \to -\infty$ along this end. By the claim and its proof, we have lower volume bounds of small $r$-balls at $p_i$ and curvature bounds on $M$ (which imply bounds on the derivative of the curvature by Shi's estimates) along this end. Thus, we can take (up to a subsequence) a Cheeger-Gromov limit of $(M, g, p_i, \nabla f)$ to get a cohomogeneity one gradient expanding soliton $(M_\infty, g_\infty, p_\infty, \nabla f_\infty)$ with topology $\mathbb{R} \times S^2 \times S^1$ (by the monotonicity of $a$ and $b$, and by rescaling $a$ by a constant if necessary, the orbits stay bounded in diameter).
		
		\medskip
		
		\noindent By the monotonicity of $f_\infty$, since we chose a sequence $r_i \to -\infty$, we see that $f_\infty$ must be bounded from below. From the beginning of the proof, we also know that $f_\infty$ is bounded from above. This implies that $f'_\infty =f''_\infty=0$, and that $f_\infty$ is thus constant. This means that $(M_\infty, g_\infty, p_\infty, f_\infty)$ is an Einstein manifold on $\mathbb{R}\times S^2 \times S^1$ with two ends, which we have ruled out by \Cref{lem:noeinstein}. Thus, there are no such two ended expanding solitons.
	\end{proof}
	
	\noindent Using the technical lemmas above, we can prove the following improved version of \Cref{lem:symcone}.
	
	\begin{lem}\label{lem:uniformly_epsilon_conical}
		Fix an $\epsilon \geqslant 0$ and suppose that $(M,g_i, \nabla f_i)$ is a sequence of gradient expanding Ricci solitons respectively asymptotic to cones over links $(S^2 \times S^1,h_i=(a'_{\infty, i})^2g_{S^1}+(b'_{\infty,i})^2g_{S^2})$. Suppose we have the following bounds for all $i$:
		
		\begin{equation*}
			\min \{ a'_{\infty,i}, b'_{\infty,i} \} \geqslant c_1 \quad \quad b'_{\infty,i} \leqslant c_2
		\end{equation*}
		
		\noindent for some constants $c_1, c_2 > 0$. Then, we have the following:
		
		\begin{enumerate}
			\item The scalar curvature satisfies $R_{g_i} \leqslant C(c_1)$ on $M$ for some constant $C(c_1) > 0$.
			\item There exists $r_0>0$ so that $(M,g_i, \nabla f_i)$ is uniformly $\epsilon$-conical at a distance $r_i \leqslant r_0$ from the tip.
		\end{enumerate}
	\end{lem}
	
	\begin{proof}
		\noindent Part 1 follows from the proof of Proposition 4.32 of \cite{BC23}. To summarize briefly, first note that the bound $\min \{ a'_{\infty,i}, b'_{\infty,i} \} \geqslant c_1$ in the hypothesis provides bounds on the injectivity radii and curvatures of the links in the form $|\text{Rm}|_{h_i} \leqslant Q \equiv Q(c_1)$ and $\text{inj}({h_i}) \geqslant \alpha \equiv \alpha(c_1)$. In the first several paragraphs of the proof of Proposition 4.32 in \cite{BC23}, it is shown via the maximum principle that the quantity $R_{g_i}+|\nabla f_i|^2 = -f_i$ is bounded from above in the form $-f_i \leqslant C(\alpha, Q)$. This provides the uniform scalar curvature upper bound on $(M,g_i)$. 
		
		\medskip
		
		\noindent We remark that we make no assumptions about the sign of of the scalar curvature, unlike in \cite{BC23}, in which it is used to prove a non-collapsing result. In a later section, we show that the assignment of the asymptotic cone to a cohomogeneity one soliton is a proper map irrespective of whether the solitons have nonnegative scalar curvature.
		
		\medskip
		
		\noindent 2. From \Cref{lem:symcone}, we know that we have a fixed $S_0 >0$ such that $(M,g_i,\nabla f_i)$ is $\epsilon$-conical at $d_i(S_0):=r_i$.

		\medskip

		\noindent Suppose the sequence $\{r_i\}$ is unbounded. Let $p_i \in (M, g_i)$ be a sequence of points with $r(p_i)=r_i$. Then, by \Cref{def:epsilonconical}, for sufficiently small $\epsilon$, we know by $\epsilon$-conicality that $|\text{Rm}|_{g_i}(p_i)$ is close to $|\text{Rm}|_{\gamma_i}(S_0,z)$ for all $i$. By the hypotheses that $ c_2 \geqslant b'_{\infty, i} \geqslant c_1$, $|\text{Rm}|_{\gamma_i}(S_0,z)$ is bounded from above and below for all $i$, implying the same conclusion for $|\text{Rm}|_{g_i}(p_i)$. Additionally, using \Cref{eq:a_estimate} and \Cref{eq:b_estimate}, we see that
		
		\begin{align*}
			a'_{\infty,i}(1-\epsilon)^{1/2}S_0 &\leqslant a_i(r_i) \\
			b'_{\infty,i}(1-\epsilon)^{1/2}S_0 &\leqslant b_i(r_i) \leqslant b'_{\infty,i}(1+\epsilon)^{1/2}S_0 \label{eq:b_estimate}
		\end{align*}

		\noindent From the above inequalities, and the hypotheses, we see that $a_i(r_i)$ is bounded from below while $b_i(r_i)$ is bounded from both above and below.
		
		\medskip
		
		\noindent Now, define \[
		\bar{M}_i := \left\{ x \in (M, g_i) \mid r(x) \leqslant r_i \right\}
		\]
		
		\medskip
		
		\noindent For any $D_i>0$, by \Cref{lem:pp} (applied to the manifold with boundary \( \bar{M}_i \)), we can choose a point \( q_i \) satisfying  
		\[
		d(p_i, q_i) \leqslant \frac{2D_i}{\sqrt{|\text{Rm}|_{g_i}(p_i)}}
		\quad \text{with} \quad r(q_i) \leqslant r_i,
		\]
		so that for \( Q_i := |\text{Rm}|_{g_i}(q_i) \), we have  
		\[
		Q_i \geqslant |\text{Rm}|_{g_i}(p_i)
		\quad \text{and} \quad
		|\text{Rm}|_{g_i} \leqslant 4 Q_i \quad \text{on} \quad B_{g_i}(q_i, D_i/\sqrt{Q_i}).
		\]
		
		\noindent
		Note: a priori, this bound only holds on the intersection of this ball with the complete metric space \( \bar{M}_i \); however, by \( \epsilon \)-conicality, for sufficiently small \( \epsilon > 0 \) and for \( r > r_i \), the curvature cannot be much larger than \( |\text{Rm}|(p_i) \), which is less than $Q_i$. Thus, the bound $\text{Rm}|_{g_i} \leqslant 4 Q_i$ holds on the entire ball $B_{g_i}(q_i, D_i/\sqrt{Q_i})$.
		
		\medskip
		
		\noindent
		Now choose a sequence \( \{D_i\} \to \infty \) such that \( r(q_i) \to \infty \). Note that  
		\[
		r(q_i) \geqslant r_i - \frac{2D_i}{\sqrt{\text{Rm}_{g_i}(p_i)}},
		\]
		by the result of point-picking, so it is possible to choose such a sequence \( D_i \). Then, rescale to get $\tilde{g}_i = Q_i g_i$ satisfying $|\text{Rm}|_{\tilde{g}_i} \leqslant 4$ on $B_{\tilde{g}_i}(q_i, D_i)$. Then, for any $D>0$, we have $D < D_i$ for large $i$, so $|\text{Rm}|_{\tilde{g_i}} \leqslant C(D)$ on $B_{\tilde{g}_i}(q_i, D)$. By Shi's estimates (which can be applied, since $\{Q_i\}$ is bounded from below), we have bounds on the derivatives of the curvature of $\tilde{g}_i$ as well on the ball. By \Cref{lem:bbound}, we can ensure lower bounds on the sizes of the $S^2$ orbits at $q_i$. We may also rescale $a$ by a constant if necessary to ensure that the sizes of the $S^1$ orbits remain bounded as well (note that this does not affect the soliton equations except at $r=0$).
		
		\medskip

		\noindent Now, consider the case where (a subsequence of) $Q_i$ is unbounded from above. By Part 1, we have a uniform scalar curvature bound on $(M,g_i)$ in terms of the curvature bound of the links. Thus, by the previous paragraph, we can consider (up to a subsequence) the Cheeger-Gromov limit of $(M,\tilde{g}_i, q_i)$ to get $(M_\infty, g_\infty, q_\infty)$. If the sizes of the $S^1$ and $S^2$ orbits containing $q_i$ in $(M,g_i)$ remain bounded, we have $M_\infty \equiv \mathbb{R} \times S^2 \times S^1$ topologically, which has two ends, since we assumed that $D_i$ was chosen so that $r(q_i) \to \infty$. If the $S^1$ ($S^2$) orbits become unbounded, we may replace $M_\infty$ by the quotient space $M_\infty/\mathbb{Z}$ ($M_\infty/\mathbb{Z}^2$), which has two ends. 
		
		\medskip
		
		\noindent The uniform scalar curvature bound implies that $(M_\infty, g_\infty)$ is Ricci-flat. Then, $(M_\infty, g_\infty)$ splits as a product of a line and a $3$-dimensional Ricci-flat manifold, implying that $M_\infty$ is flat. However, this is a contradiction to $|\text{Rm}|_{g_\infty}(q_\infty) \neq 0$.

		\medskip
		
		\noindent Now, consider the case where $Q_i$ is bounded. By the proof of Lemma 4.32(b) of \cite{BC23}, $|\nabla f|$ is uniformly bounded on $(M,g_i)$ for $r \leqslant r_i$. Using the soliton equation, the curvature bounds provide higher derivative bounds on $f_i$ (we may add an appropriate constant to each $f_i$ to ensure that $f_i(q_i)=0$). Thus, we can take (up to a subsequence) the Cheeger-Gromov limit of the solitons $(M,g_i, \nabla f_i, q_i)$ to get a certain cohomogeneity one gradient expanding soliton ($M_\infty, g_\infty, \nabla f_\infty, q_\infty)$. As in the previous case, we can quotient $M_\infty$ by translations if necessary to ensure that it has two ends.
		
		\medskip
		
		\noindent Thus, we have a cohomogeneity one gradient expanding soliton with two ends. Note that this gradient expanding soliton satisfies the monotonicity properties $a', b' \geqslant 0$ and $f', f'' \leqslant 0$. By \Cref{lem:no2ends}, we know that such solitons do not exist, leading us to a contradiction. Thus, $\{r_k\}$ is bounded, so choosing $r_0 = \text{max }r_k$, we see that $(M,g_k,\nabla f_k)$ is uniformly $\epsilon$-conical at $r_k \leqslant r_0$.
	\end{proof}	
	
	\noindent Thus, if we know that $a'_\infty$ and $b'_\infty$ are uniformly bounded above and below, we see that any cohomogeneity one gradient expanding solitons asymptotic to such cones must be $\epsilon$-conical within some fixed distance of the tip. Thus, given an $\epsilon>0$, we can think about each expanding soliton $(M,g, \nabla f)$ as a union of two regions; an $\epsilon$-conical region consisting of points where $r > r_0$ (these points are sufficiently far from the singular orbit at $r=0$), and a compact set of points with $r \leqslant r_0$.

	\medskip
	
	\noindent Now that we understand how the distance at which a soliton looks asymptotic to a cone depends on the geometry of the link of the cone, we define a map that essentially assigns to a soliton its asymptotic cone.
	
	\begin{defn}\label{def:F}
		Suppose $(a,b,f)$ is a solution to equations \cref{eq:feq,eq:aeq,eq:beq}, with boundary conditions either \cref{S1R3a,S1R3b} or \cref{S2R2a,S2R2b}. Let the corresponding soliton metric $g = dr^2 + a(r)^2g_{S^1}+b(r)^2g_{S^2}$ be asymptotic to the cone metric $\gamma = ds^2 + (a'_\infty s)^2g_{S^1} + (b'_\infty s)^2g_{S^2}$. Then, we define the map $F: (0, \infty) \times (0, \infty) \to (0, \infty) \times (0, \infty)$ as 
		
		\begin{center}$F(a_0, -f_0) = (a'_\infty, b'_\infty)$\end{center} in the case of $S^1 \times \mathbb{R}^3$ topology (boundary conditions \cref{S1R3a,S1R3b}), and 
		
		\begin{center}$F(b_0, -f_0) = (a'_\infty, b'_\infty)$\end{center} in the case of $S^2 \times \mathbb{R}^2$ topology (boundary conditions \cref{S2R2a,S2R2b}).
	\end{defn}
	
	\noindent The definition above makes sense, as the soliton metrics in the $2$-parameter families were verified to be asymptotic to cone metrics in Section 6.
	
	\medskip
	
	\noindent Having established certain facts about $\epsilon$-conicality of solitons so far in this section, we will analyze the map $F$ and show that it has good properties in the rest of the paper.
	
	\begin{thm}\label{thm:Fcont}
		In the case of either topology, the map $F$ as in \Cref{def:F} above is continuous.
	\end{thm}

	\begin{proof}
		This is a restatement of the second statement of \Cref{lem:limcont}.

	\end{proof}
	
	\noindent Now that we know that $F$ is a continuous function, it is natural to attempt to count, with sign, the number of expanding solitons which are asymptotic to a given cone. To do this, first we need to prove that $F$ is a proper map. In Section 8, we will study the behavior of $F$ as the initial conditions approach their extreme values, and use this to prove in Section 9 that $F$ is proper in the case of each topology.

	\section{Asymptotic Behavior Near Extreme Values}
	
	We have established in the previous sections that each of the expanding solitons in the $2$-parameter family over $S^2 \times \mathbb{R}^2$ or $S^1 \times \mathbb{R}^3$ are asymptotic to cones, with the cones defined by their limiting slopes $a'_\infty, b'_\infty$. Additionally, we know that these slopes are continuous functions of the initial conditions. Now, we will investigate what happens to the solitons as well as their asymptotic cones as the initial conditions tend to their extreme values. This will be used in the following section to show that the map $F$ defined in the previous section is proper.
	
	\medskip
	
	\noindent We will begin with the slightly simpler $S^1 \times \mathbb{R}^3$ case first. In this case, the two parameters are $a(0) \in (0, \infty)$ and $f''(0) \in (-\infty,0)$. The next lemma shows that the dependence on $a(0)$ is very simple.
	
	\begin{lem}\label{lem:a0dep}
		Suppose that the functions $a,b$ and $f$ satisfy the soliton equations with parameters $a(0)=a_0$ and $f''(0)=f_0$. Then, for any constant $c>0$ the solution to the soliton equations with parameters $a(0)=ca_0$ and $f''(0)=f_0$, the solution to the soliton equations is given by the triple $(ca,b,f)$.
	\end{lem}
	
	\begin{proof}
		This follows from the simple observation that the transformation $a \to ca$ leaves \cref{eq:feq,eq:aeq,eq:beq} unchanged.
	\end{proof}
	
	\noindent Thus, the dependence of the slopes on the parameter $a_0$ in the $\mathbb{S}^1 \times \mathbb{R}^3$ case is essentially trivial. With $F$ as in \Cref{def:F}, in this case, we see that if $F(a_0, -f_0) = (a'_\infty, b'_\infty)$, then $F(ca_0, -f_0) = (ca'_\infty, b'_\infty)$. In Section 9, we reduce this to one variable function $F_1(-f_0) := F(1,-f_0)$
	
	\medskip

	\noindent Now, we investigate what happens when the parameter $f_0$ tends towards the extreme of $f_0=0$. 
	
	\medskip
	
	\noindent Intuitively, as $f_0$ gets closer and closer to $0$, the soliton potential $f$ grows more slowly near $0$, taking a longer region until it becomes asymptotic to $-\frac{1}{2}r^2$. By \Cref{lem:fzero}, we see that when $f''(0)=0$, the corresponding $(M, g)$ is an Einstein manifold. In the $S^1 \times \mathbb{R}^3$ case, we can describe this Einstein manifold in even more detail.

	\begin{lem}\label{lem:hyperbolic_space}
		Consider the soliton equations \cref{eq:feq,eq:aeq,eq:beq} with $f''(0)=0$ and boundary conditions \cref{S1R3a,S1R3b} so that the manifold is diffeomorphic to $S^1 \times \mathbb{R}^3$. Then, the corresponding Riemannian manifold $(M, dr^2 + a(r)^2g_{S^1}+b(r)^2g_{S^2})$ is a quotient of the hyperbolic model space with constant negative sectional curvature equal to $-1/3$.
	\end{lem} 
	
	\begin{proof}
		
		From \Cref{lem:fzero}, we know that $f \equiv 0$ on $[0,\infty)$. This implies that $\text{Ric}_g + g=0$, or that $(M,g)$ is an Einstein manifold with negative scalar curvature. Then, the soliton equations reduce to the following:
		
		\begin{equation*}
			\displaystyle  \frac{a''}{a}+2\frac{b''}{b}=1
		\end{equation*}
		
		\begin{equation*}
			\displaystyle a'' = -2\frac{a'b'}{b}+a
		\end{equation*}
		
		\begin{equation*}
			\displaystyle b'' = \frac{1-(b')^2}{b}-\frac{a'b'}{a}+b
		\end{equation*}
		
		\noindent Now, it is easy to verify that the functions $a(r) = a_0 \text{cosh}\left(\dfrac{r}{\sqrt{3}}\right)$ and $b(r) = \sqrt{3}\text{sinh}\left(\dfrac{r}{\sqrt{3}}\right)$ satisfy the given equations as well as boundary conditions \cref{S1R3a,S1R3b}. Thus, they must be the unique solutions to the soliton equation in the case $f''(0)=0$. Up to rescaling, this metric describes the hyperbolic model space $(\mathbb{R}^4,g_H)$ in cylindrical coordinates.\end{proof}

	\noindent Next, we will show that as $f_0 \to 0$, the slopes $a'_\infty$ and $b'_\infty$ tend to $\infty$ in both the $S^1 \times \mathbb{R}^3$ and $S^2 \times \mathbb{R}^2$ cases.

	\medskip
	
	\noindent Recall the quantities $A:=\frac{a'}{a}$ and $B:=\frac{b'}{b}$ from the proof of \Cref{lem:Rmbound2}. Using these quantities, we will rewrite the soliton equations in a slightly more convenient form for this analysis, beginning with the following lemma.
	
	\begin{lem}\label{lem:ABequations}
		Consider a gradient expanding soliton $(M,g)$ with $(a,b,f)$ satisfying equations \cref{eq:feq,eq:aeq,eq:beq} and boundary conditions either \cref{S1R3a,S1R3b} or \cref{S2R2a,S2R2b}.  For $A$ and $B$ as defined above, the following equations hold on $[1,\infty)$:
		
		\begin{equation}\label{eq:Aeq}
			A'=-A^2-2AB+Af'+1
		\end{equation}
		\begin{equation}\label{eq:Beq}
			B'=\frac{1}{b^2}-2B^2-AB+Bf'+1
		\end{equation}
		
	\end{lem}
	\begin{proof}
		This follows directly from the definitions of $A$ and $B$ and equations \cref{eq:feq,eq:aeq,eq:beq}.
	\end{proof}
	
	\begin{lem}\label{lem:r0}
		Consider a cohomogeneity one gradient expanding Ricci soliton $(M,g)$ with $(a,b,f)$ satisfying equations \cref{eq:feq,eq:aeq,eq:beq} and boundary conditions either \cref{S1R3a,S1R3b} or \cref{S2R2a,S2R2b}. Then, there exists a unique $r_0 \equiv r_0(a_0,f_0)$ or $r_0(b_0,f_0)$ such that $f'(r_0) = -1$. Suppose that the other initial condition ($a_0$ or $b_0$) lies in a compact interval of the form $[\delta, 1/\delta]$ for some $\delta \in (0,1)$. As $f_0$ approaches $-\infty$, $r_0$ approaches $\infty$.
	\end{lem}
	
	\begin{proof}
		We know from \Cref{lem:fmon} and \Cref{thm:fpbound} that $f'$ is monotonically decreasing and unbounded. Thus, for every $f_0 < 0$ and $a_0$ or $b_0$ lying in $[\delta, 1/\delta]$, there exists a unique $r_0 \equiv r_0(a_0,f_0)$ or $r_0(b_0,f_0)$ such that $f'(r_0) = -1$. By the continuity of solutions to \cref{eq:feq,eq:aeq,eq:beq} in the initial conditions, and by the fact (\Cref{lem:fzero}) that $f_0=0$ implies that $f \equiv 0$ on $\mathbb{R}^+$, we know that as $f_0$ approaches $0$ while the other initial condition $(a_0$ or $b_0)$ lies in a compact set, $r_0$ tends to $\infty$. 
	\end{proof} 
	
	\noindent From now on, we will confine the initial condition ($a_0$ or $b_0$) to lie in a compact set, and consider values of $f_0$ sufficiently close to $0$ so that the corresponding $r_0$ as defined above is greater than $1$.
	
	\begin{lem}\label{lem:collectbounds}
		Consider a cohomogeneity one gradient expanding Ricci soliton $(M,g)$ with $(a,b,f)$ satisfying equations \cref{eq:feq,eq:aeq,eq:beq} and boundary conditions either \cref{S1R3a,S1R3b} or \cref{S2R2a,S2R2b}. Suppose the initial condition ($a_0$ or $b_0$) lies in a compact set $[\delta, 1/\delta]$ for some $0 < \delta < 1$. Fix a value $f^*_0 > 0$ so that for $f_0 \in (-f^*_0, 0)$, and for $a_0$ or $b_0$ in $[\delta, 1/\delta]$, the corresponding $r_0$ from \Cref{lem:r0} is greater than $1$. Then, the following inequalities hold: 
		\begin{equation}\label{eq:A1}
			0 < \alpha_1 \leqslant A(1) \leqslant \alpha_2
		\end{equation}
		\begin{equation}\label{eq:B1}
			0 < \beta_1 \leqslant B(1) \leqslant \beta_2
		\end{equation}
		\begin{equation}\label{eq:a1}
			0 < a_1 \leqslant a(1) \leqslant a_2
		\end{equation}
		\begin{equation}\label{eq:b1}
			0 < b_1 \leqslant b(1) \leqslant b_2
		\end{equation}
		\begin{equation*}
			0 \leqslant 3+\frac{2}{b(1)^2} \leqslant C_0
		\end{equation*}
		
		\noindent for positive constants $a_1,b_1, a_2, b_2, \alpha_1, \beta_1, \alpha_2, \beta_2, C_0$ whose values depend only on $f^*_0$ and $\delta$. 
	\end{lem}
	
	\begin{proof}
		As explained in Section 2, equations \cref{eq:feq,eq:aeq,eq:beq} are continuous in the initial conditions, including when $f_0 = 0$. By \Cref{lem:fzero}, $f_0 = 0$ implies that $f\equiv 0$. In this case, replicating the proof of \Cref{lem:abmon}, we see that the inequalities $a', b' > 0$ on $(0, \infty)$ continue to hold. Thus, we have $A(1), B(1) > 0$ for each $f_0 \in [-f^*_0, 0]$, proving \Cref{eq:A1} and \Cref{eq:B1} by compactness of this interval. Similarly, we have $a(1), b(1) >0$, for each $f_0 \in [-f^*_0, 0]$, which proves \Cref{eq:a1} and \Cref{eq:b1} again by compactness. The last inequality follows from (8.6). \end{proof}

	\medskip

	\noindent Now, we proceed to analyze \Cref{eq:Aeq} and \Cref{eq:Beq}. On the interval $[1,r_0]$, we have the following:
	
	\begin{equation*}
		1-A^2-2AB \geqslant A'\geqslant 1-A^2-2AB-A
	\end{equation*}
	\begin{equation*}
		1+\frac{1}{b(1)^2}-2B^2-AB \geqslant B' \geqslant 1-AB-2B^2-B
	\end{equation*}
	
	\noindent From this, we can see that the following inequalities hold: 
	
	\begin{equation}\label{eq:Aplus2B}
		C_0-(A+2B)^2 \geqslant (A+2B)' \geqslant 3-(A+2B)^2-(A+2B)
	\end{equation}
	
	\medskip
	
	\noindent Then, we can show that $(A+2B)$ is bounded in the following manner:
	
	\begin{claim}\label{claim:Aplus2B}
		Consider the setup of \Cref{lem:collectbounds}. Then, there exist constants $c, C > 0$ (depending only on $f^*_0$ and $\delta$) such that the following bounds hold on $[1,r_0]$:
		
		\begin{equation*}
			c \leqslant (A+2B) \leqslant C
		\end{equation*}
	\end{claim}

	\begin{proof}
		\noindent \textbf{Step 1:} By \Cref{eq:A1} and \Cref{eq:B1}, we have lower and upper bounds at $r=1$ of the form 
		
		\begin{center}
			$\alpha_1 + 2\beta_1 \leqslant (A+2B)(1) \leqslant \alpha_2 + 2\beta_2$
		\end{center}
		
		\noindent where the bounds depend only on $f^*_0$ and $\delta$.
		
		\bigskip
		
		\noindent \textbf{Step 2:} Now, consider the equalities corresponding to inequalities \Cref{eq:Aplus2B}:
		
		\begin{equation*}
			u'+u^2 =C_0 \hskip 3cm v'+v^2 + v =3
		\end{equation*}
		\begin{equation*}
			(A+2B)'+(A+2B)^2 \leqslant C_0 \hskip 2cm (A+2B)'+(A+2B)^2+(A+2B) \geqslant 3
		\end{equation*}
		
		\noindent where $u$ and $v$ have initial condition $u(1)=v(1)=(A+2B)(1)$. From these two equations, we see that if $u(s) = (A+2B)(s)$ for any $s \in [1,r_0]$, then $u'(s) \geqslant (A+2B)'(s)$, implying that $u \geqslant (A+2B)$ on $[1,r_0]$. By analyzing the equation for $v$ as well, we have $v \leqslant (A+2B) \leqslant u$ on $[1,r_0]$. 
		
		\medskip
		
		\noindent First, consider the equation for $u$. This IVP can be solved exactly and the solution $u$ is asymptotic to $\sqrt{C_0}$. If $u(1)<\sqrt{C_0}$, then $u$ increases and becomes asymptotic to $\sqrt{C_0}$, and if $u(1)>\sqrt{C_0}$, then $u$ decreases to $\sqrt{C_0}$. The equation for $v$ can be analyzed to show that $v$ exhibits similar behavior (for a different asymptotic constant value $v_\infty$). Both equations have solutions asymptotic to positive constants $u_\infty$ and $v_\infty$ respectively, and the constant $u_\infty$ depends on $C_0$, which itself depends on $f^*_0$ and $\delta$. Thus, we have $v \geqslant \text{min}\{v(1), v_\infty\}$ and $u \leqslant \text{max}\{u(1), u_\infty\}$ on $[1,r_0]$.
		
		\medskip
		
		\noindent \textbf{Conclusion:} On $[1,r_0]$, since we have
		
		\begin{equation*}
			\text{min}\{v(1), v_\infty\} \leqslant v \leqslant A+2B \leqslant u \leqslant \text{max}\{u(1), u_\infty\}
		\end{equation*}
		
		\noindent the bounds from Step 1 show that we can set $c = \text{min}\{\alpha_1 + 2\beta_1, v_\infty\}$ and $C = \text{max}\{\alpha_2 + 2\beta_2, u_\infty\}$, proving the lemma.\end{proof}

	\noindent The following lemma establishes lower bounds on $A$ and $B$ on $[1,r_0]$.
	
	\begin{lem}\label{lem:ABlowerbounds}
		Consider the setup of the \Cref{lem:collectbounds}. Then, there exist constants $\alpha, \beta > 0$ (depending on $f^*_0$ and $\delta$) such that the bounds $A \geqslant \alpha$, $B \geqslant \beta$ hold on $[1,r_0]$.
	\end{lem}
	
	\begin{proof}
		
		\noindent Consider the following inequalities (derived from \Cref{eq:Aeq} and \Cref{eq:Beq}) on $[1,r_0]$: 
		
		\begin{equation*}
			A' \geqslant 1-A(A+2B+1)
		\end{equation*}
		\begin{equation*}
			B' \geqslant 1-B(A+2B+1)
		\end{equation*}
		
		\noindent Using the upper bound on $(A+2B)$ from \Cref{claim:Aplus2B}, we have the inequality on $[1,r_0]$:
		
		\begin{equation*}
			A' \geqslant 1 - (C+1)A
		\end{equation*}
		
		\noindent The method of analysis is similar to that of \Cref{claim:Aplus2B}. Consider the associated ODE
		
		\begin{center}$w'(r) = 1 - (C+1)w(r)$\end{center}
		
		\noindent with $w(1)=A(1)$. We see that $A \geqslant w$ on $[1,r_0]$.
		
		\medskip
		
		\noindent Solving the ODE, we see that
		
		\begin{center}
			$\displaystyle w(r) = \frac{(A(1)(C+1)-1)e^{(C+1)(1-r)}+1}{C+1}$ 
		\end{center}
		
		\noindent so $w$ is asymptotic to a constant $w_\infty=\frac{1}{C+1}$. Thus, we see that $A \geqslant \text{min}\{A(1), w_\infty\}$. Similarly, we see that $B \geqslant \text{min}\{B(1), w_\infty\}$. Note that $w_\infty$ depends only on $C$ from \Cref{claim:Aplus2B} (which depends only on $f^*_0$ and $\delta$), while the bounds $A(1) \geqslant \alpha_1$ and $B(1) \geqslant \beta_1$ from \Cref{eq:A1} and \Cref{eq:B1}, respectively, are also dependent only on $f^*_0$ and $\delta$. Thus, we can set $\alpha = \text{min}\{\alpha_1, w_\infty\}$ and $\beta = \text{min}\{\beta_1, w_\infty\}$, making the lemma true.
	\end{proof}

	\noindent Using \Cref{lem:ABlowerbounds}, for any value of $f_0$, we can show that the slopes $a'$ and $b'$ grow exponentially on the corresponding interval $[1,r_0]$.
	
	\begin{lem}\label{lem:finalbounds}
		Suppose we are in the setup of \Cref{lem:collectbounds}. Then, on the interval $[1,r_0]$, we have the following bounds:
		
		\begin{center}
			$a(r) \geqslant a_1 e^{\alpha (r-1)} \hskip 2cm a'(r) \geqslant a_1 \alpha e^{\alpha (r-1)}$
		\end{center}
		\begin{center}
			$b(r) \geqslant b_1 e^{\beta (r-1)} \hskip 2cm b'(r) \geqslant b_1 \beta e^{\beta (r-1)}$
		\end{center}
		
		\noindent where $\alpha$ and $\beta$ are as in \Cref{claim:Aplus2B} and $a_1,a_2,b_1,b_2$ are as in \Cref{eq:a1} and \Cref{eq:b1}.

	\end{lem}
	\begin{proof}
		Consider the inequality $A > \alpha$ on $[1,r_0]$. This is equivalent to the inequality
		
		\begin{center}
			$\dfrac{a'(r)}{a(r)}-\alpha > 0$
		\end{center}
		
		\noindent Integrating this inequality gives 
		
		\begin{center}
			$a(r) \geqslant a(1)e^{\alpha (r-1)}$
		\end{center}
		
		\noindent Now, since $a(1) \geqslant a_1$ from \Cref{eq:a1}, by applying the inequality $a'(r) \geqslant \alpha a(r)$, we get the result for $a$. The same procedure for $b$ gives the corresponding result.
	\end{proof}
	
	\noindent The importance of the previous lemma is in the fact that $a$ and $b$ tend to grow exponentially at least until the point $r_0$, where $f'=-1$. This shows that the soliton mimics the behavior of an Einstein manifold (corresponding to $f''(0)=0$) in a region close to the tip. As $f_0$ gets closer to $0$, $r_0$ tends to $\infty$, suggesting that the slopes $a'$ and $b'$ grow exponentially on a larger and larger interval. 
	
	\medskip
	
	\noindent From this, we will show that $a'_\infty$ and $b'_\infty$ approach $\infty$ as $f_0$ approaches $0$ (with the other initial condition, $a_0$ or $b_0$, lying in $[\delta, 1/\delta]$, as specified earlier in this section).
	
	\begin{thm}\label{thm:ic_approach_inf} Fix $\delta \in (0,1)$ and choose $f^*_0>0$ so that the corresponding $r_0$ from Lemma 8.4 is greater than $1$ for all $f_0 \in (-f^*_0, 0)$.
		\begin{enumerate}
			\item Consider cohomogeneity one gradient expanding solitons with topology $S^1 \times \mathbb{R}^3$. Suppose $a^i_0 \in [\delta, 1/\delta]$ and $f^i_0 \in (-f^*_0,0)$ such that $f^i_0$ converges to $0$. Let $F(a^i_0,f^i_0) = (a'_{\infty,i}, b'_{\infty, i})$. Then, we have
			\begin{equation*}
				a'_{\infty,i} \text{ and } b'_{\infty,i} \text{ are unbounded from below } i \to \infty.
			\end{equation*}
			
			\item Consider cohomogeneity one gradient expanding solitons with topology $S^2 \times \mathbb{R}^2$. Suppose $b^i_0 \in [\delta, 1/\delta]$ and $f^i_0 \in (-f^*_0,0)$ such that $f^i_0$ converges to $0$. Let $F(a^i_0,f^i_0) = (a'_{\infty,i}, b'_{\infty, i})$. Then, we have
			\begin{equation*}
				a'_{\infty,i} \text{ and } b'_{\infty,i} \text{ are unbounded from below } i \to \infty.
			\end{equation*}
		\end{enumerate}
		
	\end{thm}
	\begin{proof}
		The proof is essentially identical in both cases. In either case, fix a value $f_0 \in (-f^*_0, 0)$ and, depending on the topology, pick $a_0$ or $b_0$ lying in $[\delta, 1/\delta]$. Suppose $(a,b,f)$ is the solution to \cref{eq:feq,eq:aeq,eq:beq} with these initial conditions (as well as the initial conditions required to ensure the correct topology)
		
		\medskip
		
		\noindent \textbf{Step 1: }Consider the quantities $\overline{a}=\frac{a}{a(r_0)}$ and $\overline{b}=\frac{b}{b(r_0)}$, where $r_0$ is the unique point where $f'(r_0)=-1$ as defined earlier in this section. Then, the equations for these quantities become
		
		\begin{equation}\label{eq:abar}
			\overline{a}'' = -2\frac{\overline{a}'\overline{b}'}{\overline{b}}+\overline{a}'f'+\overline{a}
		\end{equation}
		
		\begin{equation}\label{eq:bbar}
			\overline{b}'' = \frac{\frac{1}{b(r_0)^2}-(\overline{b}')^2}{\overline{b}}-\frac{\overline{a}'\overline{b}'}{\overline{a}}+\overline{b}'f'+\overline{b}
		\end{equation}
		
		\noindent with initial conditions $\overline{a}(r_0)=\overline{b}(r_0)=1$ and $\overline{a}'(r_0)$ and $\overline{b}'(r_0)$ which are bounded from below by \Cref{lem:ABlowerbounds}, since 
		
		\begin{center}
			$\overline{a}'(r_0) = \frac{a'(r_0)}{a(r_0)}=A(r_0) > \alpha$ \hskip 1cm (and similarly for $\overline{b}$)
		\end{center}
		
		\noindent and from above by Claim 8.6, since 
		
		\begin{center}
			$\overline{a}'(r_0) =A(r_0) < (A+2B)(r_0) < C$ \hskip 1cm (and similarly for $\overline{b}$)
		\end{center}
		
		\noindent where all constants involved depend only on $f^*_0$ and $\delta$.
		
		\medskip
		
		\noindent \textbf{Step 2: } By \Cref{claim:Aplus2B}, we know that $0 \leqslant B \leqslant C$ on $[1,r_0]$, where $C$ only depends on $f^*_0$ and $\delta$. By an argument almost identical to that in the proof of \Cref{lem:absimplebound}, we can show that $B \leqslant C'$ on $[r_0, \infty)$ for some constant $C'$ which also only depends on $f^*_0$ and $\delta$. Additionally, by Lemma 6.3, we see that the constant $C_1$ in the  bound $f'(r) \geqslant -(r+C_1)$ can also be chosen uniformly in $f_0 \in (-f^*_0, 0)$, since $C_1 = \sqrt{-3f_0}$

		\medskip
		
		\noindent \textbf{Step 3: }Applying these strengthened inequalities to equation \Cref{eq:abar}, we see that

		\begin{equation*}
			\overline{a}'' = \overline{a}'\left(f' -2\frac{\overline{b}'}{\overline{b}} \right)+\overline{a} \geqslant -\overline{a}' \left(r+C_1+2C' \right) + \overline{a}
		\end{equation*}
		
		\noindent where the constants $C_1$ and $C'$ depend only on $f^*_0$ and $\delta$. The initial conditions are $\overline{a}(r_0)$ and $\overline{a}'(r_0)$ which are bounded from above and below by Step 1.
		
		\medskip
		
		\noindent \textbf{Step 4: } Thus, setting $\bar{C} = C_1+2C'$, we have the following inequality on $\bar{a}$ on $[r_0, \infty)$:
		
		\begin{equation*}
			\bar{a}''(r) \geqslant -(r+\bar{C})\bar{a}'(r) + \bar{a}(r)
		\end{equation*}
		
		\noindent with initial conditions $\bar{a}(r_0) = 1$ and $\bar{a}'(r_0) = A(r_0) > \alpha$. Now, by Part 2 of \Cref{lem:ODE}, we see that $\bar{a}'(r) \geqslant C_1$ on $[r_0, \infty)$ for some constant $C_1$ depending uniformly on the initial conditions of $\bar{a}(r_0)$ and $\bar{C}$. In fact, by the proof of this Lemma, we may choose 
		
		\begin{equation*}
			C_1 = \frac{\min\left\{ \dfrac{\bar{C}}{1 + r_0}, \, \alpha \right\}}{2}
		\end{equation*}
		
		\noindent Thus, we see that $\overline{a}'(r) \geqslant C_1 \equiv C_1(r_0)$ on $[r_0, \infty)$, which implies that $a'(r) \geqslant a(r_0)C_1$ on $[r_0, \infty)$, which implies that $a'_\infty \geqslant a(r_0)C_1 \geqslant  a_1 e^{\alpha (r_0-1)}C_1$ by \Cref{lem:finalbounds}. Thus, we have a lower bound on $a'_\infty$ (depending only on $f^*_0$ and $\delta$) over all values of $f_0 \in (-f^*_0,0)$ and over all values of the other initial condition in $[\delta, 1/\delta]$. 
		
		\medskip
		
		\noindent A similar argument, analyzing equation 8.9, provides the analogous lower bound on $b'_\infty$. Note that the only difference is from the term $\displaystyle \frac{1}{{b}(r_0)^2\bar{b}}$; however, this term is bounded from above by a constant $C_b$. We can apply the same procedure as for $a$ to the quantity $b+C_b$, as we did in the proof of \Cref{lem:limcont}.
		
		\bigskip

		\noindent \textbf{Step 5: } Now, we prove the final step of the theorem in the $S^1 \times \mathbb{R}^3$ case; the other case is nearly identical. Suppose $(a^i_0, f^i_0)$ is a sequence of initial conditions satisfying the hypotheses. Then, by the results of Step 4, we know that $a'_{\infty,i} \geqslant  a_1 e^{\alpha (r_{0,i}-1)}C_1(r_{0,i})$, where $r_{0,i}$ is the unique value satisfying $f'(r_{0,i})=-1$. Using Lemma 8.4, we know that $r_{0,i}$ approaches $\infty$. Thus, for sufficiently large $i$, $C_1(r_{0,i}) = \frac{\bar{C}}{1+r_{0,i}}$. Thus, we have
		
		\begin{center}
			$\displaystyle a'_{\infty,i} \geqslant \frac{a_1 \bar{C} e^{\alpha (r_{0,i}-1)}}{1+r_{0,i}}$
		\end{center}
		
		\noindent Thus, we know that $r_{0,i}$ approaches $\infty$, $a'_{\infty,i}$ becomes unbounded as well. A similar argument works for $b'_{\infty,i}$, which concludes the proof of the theorem.
	\end{proof}

	\section{Calculation of Expander Degree}
	
	\noindent In this section, we will study the relation between the initial conditions $(a_0, f_0)$ or $(b_0, f_0)$ and the slopes of the cone $(a'_\infty, b'_\infty)$. Recall from \Cref{def:F} that $F: \mathbb{R}^+ \times \mathbb{R}^+ \to \mathbb{R}^+ \times \mathbb{R}^+$ was defined as 
	
	\begin{center}
		$F(a_0, -f_0) = (a'_\infty, b'_\infty)$ \hskip 2cm (in the $S^1 \times \mathbb{R}^3$ case)

		$F(b_0, -f_0) = (a'_\infty, b'_\infty)$ \hskip 2cm (in the $S^2 \times \mathbb{R}^2$ case)
	\end{center}
	
	\noindent where $(a'_\infty, b'_\infty) = \text{lim}_{r \to \infty} (a'(r), b'(r))$, for the warping functions $a$ and $b$ in the soliton metric $g = dr^2 + a(r)^2g_{S^1}+b(r)^2g_{S^2}$ and the soliton potential $f$ satisfying \cref{eq:feq,eq:aeq,eq:beq}, with initial conditions \cref{S1R3a,S1R3b} for $S^1 \times \mathbb{R}^3$ and \cref{S2R2a,S2R2b} for $S^2 \times \mathbb{R}^2$. We will show that $F$ is proper in the case of each topology.

	\begin{thm}\label{thm:Fproper}
		\begin{enumerate}
			\item The map $F$ is proper in the case of $S^1 \times \mathbb{R}^3$ topology.
			\item The map $F$ is proper in the case of $S^2 \times \mathbb{R}^2$ topology.
		\end{enumerate}
	\end{thm}
	
	\begin{proof}
		1. In the $S^1 \times \mathbb{R}^3$ case, consider a sequence $(a'_{\infty,i}, b'_{\infty,i})$ in the range of $F$ that converges to the pair of positive real numbers $(a'_\infty, b'_\infty)$. Suppose that $(a'_{\infty, i}, b'_{\infty, i})=F(a^i_0, -f^i_0)$ for initial conditions $a^i_0$ and $f^i_0$. Suppose that the corresponding expanding solitons $(M, g_i = dr^2 + a_i(r)^2g_{S^1}+b_i(r)^2g_{S^2}, \nabla f_i)$ are asymptotic to the cones $\gamma_i$. Recall from \Cref{def:cone} and \Cref{thm:GHC} that this implies that for a point $p$ with $r(p)=0$ and any sequence $\lambda_i \to 0$, the sequence of pointed manifolds $(M, \lambda_i^2 g_i, p)$ converges in the Gromov-Hausdorff sense to a cone over the link $(S^2 \times S^1, h)$, where $h$ admits an isometric action of $\text{SO}(3) \times \text{SO}(2)$. As a consequence, the functions $a_i$ and $b_i$ are asymptotically linear with $(a'_{\infty, i}, b'_{\infty, i}) = \text{lim}_{r \to \infty} (a_i'(r), b_i'(r))$.

		\medskip

		\noindent Fix a small $\epsilon >0$. As $(a'_{\infty, i}, b'_{\infty, i})$ converges, we have uniform lower bounds $a'_\infty, b'_\infty \geqslant c$ for some $c>0$, and $b'_{\infty, i} \leqslant C$ for some $C>0$. Thus, by \Cref{lem:symcone}, we know that there is a uniform $S_0 > 0$ so that each $(M, g_i, \nabla f_i)$ is $\epsilon$-conical at distance $d_i(S_0)$ from the tip, providing the following inequality from \Cref{eq:a_estimate}:
		
		\begin{center}
			$a'_{\infty,i}S_0 (1-\epsilon)^{1/2} \leqslant a_i(d_i(S_0)) \leqslant a'_{\infty,i}S_0 (1+\epsilon)^{1/2}$

		\end{center}
		
		\noindent By Part 2 of \Cref{lem:uniformly_epsilon_conical}, we know that the sequence $d_i(S_0)$ is bounded from above by an $r_0 >0$. Additionally, by the monotonicity of $a$ from \Cref{lem:abmon}, we know that $a_i(r_0) \geqslant a_i(d_i(S_0)) \geqslant a_i(0)$.

		\medskip
		
		\noindent From equation \Cref{eq:apgrowth}, we know that $a_i' \leqslant a_i$ on $\mathbb{R}^+$. Integrating this inequality, we see that $a_i(r) \leqslant a^i_0e^r$, which gives us the inequality $a_i(r_0) \leqslant a^i_0e^{r_0}$. Combining this with the previous inequalities, we get the following for all $i$:
		
		\begin{center}
			$c S_0 (1-\epsilon)^{1/2} \leqslant a'_{\infty,i} S_0 (1-\epsilon)^{1/2} \leqslant a_i(d_i(S_0)) \leqslant a_i(r_0) \leqslant a^i_0e^{r_0}$
		\end{center}
		
		\noindent which shows that $a^i_0 \geqslant c S_0 (1-\epsilon)^{1/2}e^{-r_0}$, providing a lower bound on $a^i_0$.
		
		\medskip 
		
		\noindent For the upper bound on $a^i_0$, we know that for all $i$, we have
		
		\begin{center}
			$a^i_0 \leqslant a_i(d_i(S_0)) \leqslant a'_{\infty,i}S_0 (1+\epsilon)^{1/2} \leqslant C S_0 (1+\epsilon)^{1/2}$
		\end{center}
		
		\noindent where the last inequality follows as convergent sequences are bounded. Thus, $a^i_0$ is bounded both from above and below, so it lies in an interval of the form $[\delta, 1/\delta]$ for $\delta <1$.
		
		\medskip
		
		\noindent By Part 1 of \Cref{lem:uniformly_epsilon_conical}, we have a uniform scalar curvature upper bound on $(M,g_i, \nabla f_i)$, and by equation \cref{eq:R0}, we know that $R_i(0) = -4-3f^i_0$, so we must have that $f^i_0$ is bounded from below. By the continuity of $F$ (\Cref{thm:Fcont}) and \Cref{thm:ic_approach_inf}, and the fact that $a_0$ lies in a compact set by the previous paragraph, we know that since $a'_{\infty,i}$ and $b'_{\infty,i}$ are bounded from above, that $f^i_0$ must be bounded from above by a constant $C<0$. Thus, in this case, we know that the initial conditions $(a^i_0, f^i_0)$ lie in a compact set, and thus a subsequence of the initial conditions converges. This shows that $F$ is proper in this case.
		
		\bigskip

		\noindent 2. In the $S^2 \times \mathbb{R}^2$ case, suppose that $(a'_{\infty, i}, b'_{\infty, i})=F(b^i_0, -f^i_0)$ for initial conditions $b^i_0$ and $f^i_0$, with $(a'_{\infty, i}, b'_{\infty, i})$ converging to $(a'_{\infty}, b'_{\infty}) \in \mathbb{R}^+ \times \mathbb{R}^+$, and denote the corresponding expanding solitons by $(M, g_i = dr^2 + a_i(r)^2g_{S^1}+b_i(r)^2g_{S^2}, \nabla f_i)$. First, the proof of the lower bound on $f^i_0$ from the $S^1 \times \mathbb{R}^3$ case carries over (with the only change being that $R_i(0) =  -4 - 2f^i_0$ in the $S^2 \times \mathbb{R}^2$ case). Additionally, as in the proof of the $S^1 \times \mathbb{R}^3$ case, we can prove the upper bound on $b^i_0$ using uniform $\epsilon$-conicality. 
		
		\medskip
		
		\noindent Now, we claim that $b^i_0$ is bounded from below. Assume that this is false, and up to a subsequence, $b^i_0 \equiv \lambda_i \to 0$. Then, consider the new quantities
		
		\begin{center}
			$\displaystyle \tilde{a}_i(r) = \frac{1}{\lambda_i}a_i(\lambda_ir) \hskip 1cm \tilde{b}_i(r) = \frac{1}{\lambda_i}b_i(\lambda_ir) \hskip 1cm \tilde{f}_i(r) = f_i(\lambda_ir)$
		\end{center}
		
		\noindent which satisfy the following equations on $[0, \infty)$
		
		\begin{equation*}
			\displaystyle \tilde{f_i}'' = \frac{\tilde{a}_i''}{\tilde{a}_i}+2\frac{\tilde{b}_i''}{\tilde{b}_i}-\lambda^2_i
		\end{equation*}
		
		\begin{equation*}
			\displaystyle \tilde{a}_i'' = -2\frac{\tilde{a}_i'\tilde{b}_i'}{\tilde{b}_i}+\tilde{a}_i'\tilde{f}_i'+\lambda^2_i\tilde{a}_i
		\end{equation*}
		
		\begin{equation*}
			\displaystyle \tilde{b}_i'' = \frac{1-(\tilde{b}_i')^2}{\tilde{b}_i}-\frac{\tilde{a}_i'\tilde{b}_i'}{\tilde{a}_i}+\tilde{b}_i'\tilde{f}_i'+\lambda^2_i\tilde{b}_i
		\end{equation*}
		
		\noindent with the initial conditions
		
		\begin{align*}
			\tilde{a}_i(0)&=0 \hskip 1cm  \tilde{a}'_i(0)=1\\
			\tilde{b}_i(0)&=1 \hskip 1cm  \tilde{b}'_i(0)=0\\
			\tilde{f}_i(0)&=0 \hskip 1cm  \tilde{f}'_i(0)=0 \hskip 1cm  \tilde{f}''_i(0)=(\lambda_i)^2f^i_0\\
		\end{align*}
		
		\noindent Then, in the limit as $i \to \infty$, we have the equations
		
		\begin{equation*}
			\displaystyle f'' = \frac{a''}{a}+2\frac{b''}{b}
		\end{equation*}
		
		\begin{equation*}
			\displaystyle a'' = -2\frac{a'b'}{b}+a'f'
		\end{equation*}
		
		\begin{equation*}
			\displaystyle b'' = \frac{1-({b}')^2}{b}-\frac{{a}'{b}'}{{a}}+{b}'{f}'
		\end{equation*}
		
		\noindent with initial conditions 
		
		\begin{align*}
			{a}(0)&=0 \hskip 1cm  {a}'(0)=1\\
			{b}(0)&=1 \hskip 1cm  {b}'(0)=0\\
			f(0)&=0 \hskip 1cm    {f}'(0)=0 \hskip 1cm  f''(0)=0\\
		\end{align*}
		
		\noindent since $f^i_0$ is negative and bounded from below. By an argument nearly identical to Part 2 of Lemma 4.2 of \cite{A17}, we see that this implies that $f \equiv 0$ on $\mathbb{R}^+$, thus implying that $g = dr^2 + a(r)^2g_{S^1}+b(r)^2g_{S^2}$ is a Ricci-flat metric on $\mathbb{R}^2 \times S^2$. From Chapter 2 of \cite{Pet}, we know that this metric is asymptotic to the flat metric on $S^1 \times \mathbb{R}^3$ and that $a(r) \approx C$ and $b(r) \approx r$ for large $r$, where $C>0$ is a constant.
		
		\medskip
		
		\noindent Then, for any large $L > 0$, we can choose $r_0 > 0$ so that the following hold:
		
		\begin{center}
			$\displaystyle \frac{b(r_0)}{a(r_0)} \geqslant L \hskip 1cm \frac{d}{dr}\left(\frac{b(r)}{a(r)}\right)(r_0) > 0$
		\end{center}
		
		\noindent Then, for sufficiently large $i$ (depending on $\delta$), we have
		
		\begin{center}
			$\displaystyle \frac{b_i(r_0)}{a_i(r_0)} \geqslant L-1 \hskip 1cm \frac{d}{dr}\left(\frac{b_i(r)}{a_i(r)}\right)(r_0) > 0$
		\end{center}
		
		\noindent This implies that for $P_i=\frac{b_i}{a_i}$ as defined in Section 3, we have that $P_i(\lambda_i r_0) \geqslant L-1$ and $P'_i(\lambda_i r_0) \geqslant 0$. By \Cref{lem:Pgrowth}, we know that $P_i(r) \geqslant L-1$ for all $r \geqslant \lambda_i r_0$.

		\medskip
		
		\noindent We also know that $\lim_{r \to \infty} P_i(r) = b'_{\infty,i}/a'_{\infty,i} < C$ for some constant $C>0$ independent of $i$, by the convergence of $(a'_{\infty,i}, b'_{\infty,i})$ by hypothesis. However, this contradicts the conclusion of the previous paragraph by choosing $L$ to be arbitrarily large. Thus, our assumption was false, and it must be the case that $b^i_0$ is bounded from below.
		
		\medskip
		
		\noindent Then, by the continuity of $F$ and the fact that $b_0$ lies in a compact set and \Cref{thm:ic_approach_inf}, we get that $f^i_0$ must be bounded from above as well, just as in the $S^1 \times \mathbb{R}^3$ case.

		\medskip
		
		\noindent Thus, we have shown that the initial conditions lie within a compact set in this case as well, so a subsequence of the initial conditions must converge, implying that the map $F$ is proper in the $S^2 \times \mathbb{R}^2$ case.
	\end{proof}

	\noindent As a consequence of the previous theorem, we have the following:
	
	\begin{cor}\label{cor:degree}
		The degrees of the maps $F$ in the $S^1 \times \mathbb{R}^3$ case and in the $S^2 \times \mathbb{R}^2$ case are well defined.
	\end{cor}
	
	\noindent The importance of the properness of $F$ is in concluding the corollary above; in \cite{BC23}, the expander degree of an orbifold was defined on the space of gradient expanding solitons on the interior of the orbifold with positive scalar curvature, albeit in a more general and non-symmetric setting. In our cohomogeneity one setting, we do not need this assumption. 
	
	\medskip
	
	\noindent Analogously, we can define a cohomogeneity one version of this quantity.
	
	\begin{defn}\label{def:degexp}
		The \textbf{cohomogeneity one expander degree}, denoted $\textnormal{deg}^{\textnormal{sym}}_\textnormal{exp}$, of the orbifolds $S^1 \times \mathbb{D}^3$ and $S^2 \times \mathbb{D}^2$ are defined as the topological degree of the corresponding maps $F$ in the cases of the topologies $S^1 \times \mathbb{R}^3$ and $S^2 \times \mathbb{R}^2$, respectively.
	\end{defn}
	
	\noindent Now, we can calculate the cohomogeneity one expander degree in the case of each topology.
	
	\medskip
	
	\noindent First, we consider the $S^1 \times \mathbb{R}^3$ case. We will calculate the limit of $b'_\infty$ as $f_0 \to -\infty$.

	\begin{lem}\label{lem:f0inf}
		In the case of $S^1 \times \mathbb{R}^3$ topology, consider a sequence of initial conditions $(a^i_0, f^i_0)$. Let $(M, g_i, \nabla f_i)$ be the corresponding cohomogeneity one gradient expanding solitons asymptotic to cones $\gamma_i = dr^2 + r^2h_i$ over the link $(S^2 \times S^1, h_i = (a'_{\infty,i})^2 g_{S^1}+(b'_{\infty,i})^2 g_{S^2})$. Suppose $f^i_0 \to -\infty$, and set
		
		\begin{center}
			$(a'_{\infty, i}, b'_{\infty, i}) = F(a^i_0, f^i_0)$
		\end{center}
		\noindent Then, $b'_{\infty, i}$ converges to $0$.
	\end{lem}
	
	\begin{proof}
		Suppose the conclusion of the lemma is not true. Then, there would exist a sequence of expanding solitons $(M, g_i, \nabla f_i)$ with initial conditions $(a^i_0, f^i_0)$ and asymptotic cone metrics $\gamma_i$ such that $f^i_0 \to -\infty$ but $b'_{\infty, i} \geqslant C$ for some $C>0$. 
		
		\medskip
		
		\noindent Denote $\text{inj}_{h_i}=\alpha_i$. Note that since $b'_{\infty, i}$ is bounded below, the sequence $\{\alpha_i\}$ is bounded from below iff $a'_{\infty, i}$ is bounded below. Suppose that $\alpha_i \to 0$. Then, we can consider the new sequence of expanding solitons $(M, \tilde{g}_i, \nabla f_i)$ with initial conditions $(a^i_0/\alpha_i, f^i_0)$. Using \Cref{lem:a0dep}, we see that $(\tilde{a}_{\infty,i}, \tilde{b}_{\infty,i}) := F(a^i_0/\alpha_i, f^i_0) = (a'_{\infty, i}/\alpha_i, b'_{\infty, i})$. These solitons are respectively asymptotic to the cone metrics $\tilde{\gamma}_i = dr^2 + r^2\tilde{h}_i$, in which the injectivity radii of $\tilde{h}_i$ are uniformly bounded from below. Thus, we have $\text{min}\{\tilde{a}_{\infty,i}, \tilde{b}_{\infty,i}\} \geqslant c$, for some constant $c>0$.

		\medskip 
		
		\noindent Then, by Part 1 of \Cref{lem:uniformly_epsilon_conical}, the sequence $(M, \tilde{g}_i, \nabla f_i)$ would have uniformly bounded scalar curvature. However, we calculated in equation \cref{eq:R0} that $\tilde{R}_i(0) = -3f^i_0 -4$, which is clearly unbounded as $f^i_0 \to -\infty$, which is a contradiction. Thus, we must have that $b'_{\infty,i} \to 0$ as $f^i_0 \to 0$.
	\end{proof}
	
	\begin{thm}\label{thm:S1R3_degexp}
		$\textnormal{deg}^{\textnormal{sym}}_\textnormal{exp}(S^1 \times \mathbb{D}^3) = 1$ (up to sign)
	\end{thm}
	
	\begin{proof}
		First, by \Cref{lem:a0dep}, we see that changing $a_0$ simply scales $a'_\infty$ and leaves $b'_\infty$ invariant. Thus, we can consider the maps
		
		\begin{equation*}
			F_1 : \mathbb{R}^+ \to \mathbb{R}^+
		\end{equation*}
		
		\noindent defined in the following way: suppose $F(1, -f_0) = (a'_\infty, b'_\infty)$. Then, set $F_1(-f_0) := p_2(F(1,-f_0)) = b'_\infty$, where $p_2$ is the projection onto the second component. We have the following lemma:

		\begin{lem}\label{lem:deg_lemma}
			The degree of $F_1$ coincides with the degree of $F$.
		\end{lem}
		
		\begin{proof}[Proof of \Cref{lem:deg_lemma}]
			Since $a_0$ merely scales $a'_\infty$ and does not affect $b'_\infty$, we can write $F(a_0, -f_0) \equiv (a_0 F_0(-f_0), F_1(-f_0))$ where $F_1$ is as above and $F_0 : \mathbb{R}^+ \to \mathbb{R}^+$ is a continuous function. By \Cref{lem:f0inf}, we have that $F_1(-f_0) \to 0$ as $-f_0 \to \infty$, and by \Cref{thm:ic_approach_inf}, we know that $F_1(-f_0) \to \infty$ as $-f_0 \to 0$. Thus, $F_1$ is a proper map.
			
			\medskip
			
			\noindent Now, consider the map $H : [0,1] \times \mathbb{R}^+ \times \mathbb{R}^+ \to \mathbb{R}^+ \times \mathbb{R}^+$ given by $H(t,x,y) = ((1-t)x+txF_0(y), F_1(y))$. We will show that $H$ is a proper homotopy between $F$ and the map $(x,y) \mapsto (x, F_1(y))$. Suppose that $(t_n, x_n, y_n)$ is a sequence in $[0,1] \times \mathbb{R}^+ \times \mathbb{R}^+$ such that $H(t_n, x_n, y_n)$ converges. Then, since $F_1$ is a proper map and $F_1(y_n)$ converges, we may assume that $y_n$ is contained in a compact set $[\delta, \frac{1}{\delta}]$ of $\mathbb{R}^+$ for some $\delta \in (0,1]$. Then, as $F_0$ is continuous, we see that $\{F_0(y_n)\} \subseteq F_0([\delta, \frac{1}{\delta}]) \subseteq \mathbb{R}^+$, so $F_0(y_n)$ is bounded from above and below independently of $n$. From this, it is easy to see that $(1-t_n)+t_nF_0(y_n)$ lies in a compact subset $[c,C]$ of $\mathbb{R}^+$. Then, since $(1-t_n)x_n+t_nx_nF_0(y_n)$ converges, we see that $x_n$ also lies in a compact subset of $\mathbb{R}^+$. Finally, $t_n$ lies in a compact set by compactness of $[0,1]$. This shows that $H$ is a proper homotopy.
			
			\medskip
			
			\noindent The map $(x,y) \mapsto (x, F_1(y))$ is a product map, so its degree is $\text{deg}(x \mapsto x) \text{deg}(F_1) = \text{deg}(F_1)$. As proper homotopies preserve the degrees of continuous proper maps, we have that $\text{deg}(F) = \text{deg}(F_1)$.
		\end{proof}

		\noindent \textit{Proof of \Cref{thm:S1R3_degexp} cont.} By \Cref{lem:deg_lemma}, it is enough to compute $\text{deg}(F_1)$. Consider the map $H: [0,1] \times \mathbb{R^+} \to \mathbb{R}^+$, given by $H(s,x) = (1-s)F_1(x)+\frac{s}{x}$. It is straightforward to check that $H$ is a proper homotopy between $F_1$ and the map $x \to \frac{1}{x}$ on $\mathbb{R}^+$. Thus, as proper homotopies preserve degree, it is clear that $\text{deg}(F_1) = \text{deg}(x \to \frac{1}{x}) = -1$, which is the same as $1$ up to sign.
	\end{proof}

	\noindent \textbf{Remark:} Similar methods are employed in \cite{NW24} in their construction of expanders on $\mathbb{R}^3 \times S^1$. In the notation of that paper, a proof that $\sigma_2$ is continuous along with a properness result would constitute a proof of the previous theorem.
	
	\medskip
	
	\noindent Next, we consider the $S^2 \times \mathbb{R}^2$ case.
	
	\begin{thm}\label{thm:S2R2_degexp}
		$\textnormal{deg}^{\textnormal{sym}}_\textnormal{exp}(S^2 \times \mathbb{D}^2) = 0$
	\end{thm}
	
	\begin{proof}
		We will show that $F$ is not surjective; this is sufficient to prove that the degree is $0$. Consider the set $S$ defined as
		\[
		S = \left\{ (b_0, -f_0) \subseteq \mathbb{R}^+ \times \mathbb{R}^+ \mid b'_\infty = 1,  \text{ where } (a'_\infty, b'_\infty) = F(b_0, -f_0) \right\}.
		\]
		
		\noindent We claim that over all initial conditions in $S$, the value of $a'_\infty$ is bounded from above. Suppose that this is not true. Then, there would exist a sequence of initial conditions $(b^i_0, f^i_0)$ with $F(b^i_0, -f^i_0) = (a'_{\infty, i}, b'_{\infty, i})$ satisfying $b'_{\infty,i} = 1$ and $a'_{\infty,i} \to \infty$. Then, from equation \cref{eq:aeq}, we have the inequality $a_i'' \leqslant a_i$, with the initial conditions $a_i(0) = 0, a_i'(0) = 1$. Integrating this inequality (using equation \Cref{eq:apgrowth} and the monotonicity of $a$), we see that $a_i(r) \leqslant \text{sinh}(r)$ for all $i$.

		\medskip
		
		\noindent Now, we clearly have bounds of the form 
		
		\begin{equation*}
			\min \{ a'_{\infty,i}, b'_{\infty,i} \} \geqslant c_1 \quad \quad b'_{\infty,i} \leqslant c_2
		\end{equation*}
		
		\noindent for some constants $c_1, c_2 > 0$. Thus, from \Cref{lem:symcone}, we know that there exists an $S_0 >0$ so that each soliton $(M, g_i)$ is uniformly $\epsilon$-conical at a distance $\{d_i(S_0)\}$ from the tip. Additionally, by Part 2 of \Cref{lem:uniformly_epsilon_conical}, the sequence $d_i(S_0)$ is bounded above by a positive $r_0$. By inequality \Cref{eq:a_estimate} in the statement of \Cref{lem:symcone} and the monotonicity of $a$ in \Cref{lem:abmon}, this implies that

		\begin{center}
			$a'_{\infty,i} S_0(1-\epsilon)^{1/2} \leqslant a_i(d_i(S_0)) \leqslant a_i(r_0)$
		\end{center}
		
		\noindent Then, as $i \to \infty$, we have that $a'_{\infty,i}$ becomes unbounded from above by hypothesis, implying the same conclusion for $a_i(r_0)$ as well. However, the bound $a_i(r_0) \leqslant \text{sinh}(r_0)$ indicates that $a_i(r_0)$ is bounded above, which is a contradiction. Thus, no such sequence of solitons $(M, g_i)$ can exist, which implies that the value of $a'_\infty$ is bounded over all solitons in $S$. Thus, there exist pairs $(a'_\infty, b'_\infty)$ which are not in the image of $F$, so it is not surjective, and thus has degree $0$.
	\end{proof}

	\appendix
	\renewcommand{\thesection}{\Alph{section}}

	\setcounter{section}{0}
	\section*{Appendix A: Derivation of Soliton Equations}
	\refstepcounter{section}

	Consider a metric $g$ on a $4$-manifold $M$ of the form
	
	\begin{center}
		$g=dr^2+a(r)^2g_{S^1}+b(r)^2g_{S^2}$
	\end{center}
	
	\noindent as in Section 2.
	
	\medskip
	
	\noindent Choose a local orthonormal frame $e^1 = dr$, $e^2= ad\theta$, and $e^i = b\hat{e}^i$ where the $\hat{e}^i$ form an orthonormal basis for $S^2$ for $i=3,4$. Denote the dual vector fields by $E_i$. In this frame, one can compute the nonzero components of the curvature to be:
	
	\begin{center}
		$\displaystyle \text{Rm}_{1221} = -\frac{a''}{a} \hskip 1cm \text{Rm}_{1331} = -\frac{b''}{b} \hskip 1cm \text{Rm}_{1441} = -\frac{b''}{b}$ 
		
		\smallskip
		
		$\displaystyle \text{Rm}_{2332} = -\frac{a'b'}{ab} \hskip 1cm \text{Rm}_{2442} = -\frac{a'b'}{ab} \hskip 1cm \text{Rm}_{3443} = \frac{1-(b')^2}{b^2}$ 
	\end{center}
	
	\noindent Thus, the  nonzero components of the Ricci tensor are
	
	\begin{center}
		$\displaystyle \text{Ric}_{11} = -\frac{a''}{a}-2\frac{b''}{b}$ 
		
		\medskip
		
		$\displaystyle \text{Ric}_{22} = -\frac{a''}{a}-2\frac{a'b'}{ab}$
		
		\medskip
		
		$\displaystyle \text{Ric}_{33} = -\frac{b''}{b}-\frac{a'b'}{ab}+\frac{1-(b')^2}{b^2}$
	\end{center}
	
	\medskip
	
	\noindent Consider a smooth function $f : M \to \mathbb{R}$ which is constant along the $S^1$ and $S^2$ directions -- in other words, $f$ depends only on $r$. The nonzero components of the Hessian $\nabla^2 f$ are
	
	\begin{center}
		$\displaystyle \nabla^2f(E_1,E_1) = f''$ 
		
		\medskip
		
		$\displaystyle \nabla^2f(E_2,E_2) = \frac{a'f'}{a}$ 
		
		\medskip
		
		$\displaystyle \nabla^2f(E_3,E_3) = \nabla^2f(E_4,E_4) = \frac{b'f'}{b}$ 
	\end{center}
	
	\noindent Now, suppose that $(M,g)$ is an expanding Ricci soliton with soliton potential $f$ satisfying the equation
	
	\begin{center}
		$\displaystyle \text{Ric}_g + \nabla^2 f + g=0$
	\end{center}
	
	\noindent In the frame chosen above, the soliton equations take the form
	
	\begin{center}
		$\displaystyle -\frac{a''}{a}-2\frac{b''}{b} + f'' + 1 = 0$ 
		
		\medskip
		
		$\displaystyle -\frac{a''}{a}-2\frac{a'b'}{ab}+\frac{a'f'}{a}+1=0$ 
		
		\medskip
		
		$\displaystyle -\frac{b''}{b}-\frac{a'b'}{ab}+\frac{1-(b')^2}{b^2} + \frac{b'f'}{b} + 1  = 0$ 
	\end{center}
	
	\noindent Rearranging the three equations above gives us the soliton equations \cref{eq:feq,eq:aeq,eq:beq}
	
	\medskip
	
	\noindent From the above, we also notice a relation between $\Delta f$ and $f''$ as follows
	
	\begin{equation}\label{eq:Laplacian}
		\displaystyle \Delta f = f'' + \left(\frac{a'}{a}+2\frac{b'}{b} \right)f'
	\end{equation}
	
	\setcounter{section}{1}
	\section*{Appendix B: Existence of Local Solutions}
	\refstepcounter{section}
	
	\noindent It is not immediately clear why equations \cref{eq:feq,eq:aeq,eq:beq} have a unique solution given a value of $f''(0) \leqslant 0$. The following theorem, proven in \cite{A17} explains why this is the case:
	
	\begin{thm}\label{thm:existence}
		Let $n \in \mathbb{N}$, $c \in \mathbb{R}$ and $U$ an open subset of $\mathbb{R}^n$ containing the origin. Let
		
		\begin{center}
			$P: U \times \mathbb{R} \times \mathbb{R} \to \mathbb{R}^n$ \hskip 1cm $(u,r,\lambda) \to P(u,r,\lambda)$
		\end{center}
		
		\noindent be a vector valued analytic function around $(\vec{0},0,c)$ such that $P(0,0,\lambda)=0$ for all $\lambda \in \mathbb{R}$. If there is an open interval $I$ containing $c$ such that for all $\lambda \in I$, the matrix $\frac{\partial P}{\partial u}(\vec{0},0,\lambda)$ has no positive integer eigenvalues and
		
		\begin{center}
			$\displaystyle \text{sup}_{\lambda \in I, m \in \mathbb{N}} \left|\left| \left(mI_n - \frac{\partial P}{\partial u} \right)^{-1} \right|\right| =  B < \infty $.
		\end{center}
		
		\noindent then there exists an $\epsilon > 0$ and a one-parameter family of analytic vector-valued functions $u(\cdot, \lambda) : (-\epsilon, \epsilon) \to \mathbb{R}^n$ solving the ODE system
		
		\begin{equation}
			\displaystyle r \frac{du(r,\lambda)}{dr} = P(u(r,\lambda),r,\lambda)
		\end{equation}
		
		\begin{equation*}
			u(0,\lambda)=0
		\end{equation*}	
		
		\noindent for $\lambda \in (c-\epsilon, c+\epsilon)$. Furthermore, $u$ depends analytically on $\lambda$.
	\end{thm}
	
	\noindent The proof, given in \cite{A17}, involves the construction of a formal power series for $u$ which satisfies the system. It is shown that the series has a positive radius of convergence which establishes the existence of a local solution. Following Appleton's methods, we will transform the soliton equations into a form suitable to apply this theorem. The following is very similar to the proof of Theorem 2.1 in \cite{A17}.
	
	\medskip
	
	\noindent First, we consider the $S^2 \times \mathbb{R}^2$ case. Let $s$ denote the independent variable of the soliton equations. Note that $a'(0)=1 \neq 0$ and $a'>0$ for  $s \in (0,\infty)$, so $a$ can be chosen as the independent variable of the soliton equations under the coordinate change corresponding to
	
	\begin{center}
		$\displaystyle g = \frac{da^2}{h(a^2)}+g_{a,b(a)}$
	\end{center}
	
	\noindent Setting $r=a^2$, we see as in the proof of Theorem 2.1 in \cite{A17} that
	
	\begin{center}
		$\displaystyle \frac{dr}{ds} = 2\sqrt{rh(r)}$
	\end{center}
	
	\noindent Thus, in this case, \cref{eq:feq,eq:aeq,eq:beq} are transformed into the following (where we use $\dot{f}$ to denote $\frac{\partial f}{\partial r}$, etc.)
	
	\begin{equation}\label{eq:newf}
		\displaystyle \ddot{f} = \frac{1}{4r}\frac{\dot{h}}{h}+2\frac{\ddot{b}}{b}+\frac{1}{r}\frac{\dot{b}}{b}+\frac{\dot{b}\dot{h}}{bh}-\frac{1}{4rh}-\frac{1}{2r}\dot{f}-\frac{1}{2}\frac{\dot{h}}{h}\dot{f}
	\end{equation}
	
	\begin{equation}\label{eq:newa}
		\displaystyle \dot{h} = - 4h \frac{\dot{b}}{b} + 2h \dot{f}+1\\
	\end{equation}
	
	\begin{equation}\label{eq:newb}
		\displaystyle \ddot{b} = \frac{1}{4rhb}  - \frac{\dot{b}}{r} - \frac{1}{2} \frac{\dot{h}}{h}\dot{b} -  \frac{(\dot{b})^2}{b} + \dot{f}\dot{b}+\frac{b}{4rh}
	\end{equation}
	
	\noindent with boundary conditions
	\begin{align*}
		b(0)&=b_0 \\
		\dot{b}(0) &= \frac{1}{4}\left(b_0+\frac{1}{b_0}\right)\\
		h(0) &= 1\\
		f(0) &=0\\
		\dot{f}(0) &= \frac{f''(0)}{2} \equiv c
	\end{align*}
	
	\noindent The boundary condition $\dot{f}(0)$ was derived by applying l'H\^opital's rule to the quantity $\dot{f}(r) = \frac{f'(r)}{2a(r)a'(r)}$ and noting that $a'(0)=1$. $\dot{b}(0)$ was derived similarly, noting that $b''(0) = \frac{1}{2}(b_0+\frac{1}{b_0})$. Note that $\dot{f}(0)$ can take on any real number value.
	
	\smallskip
	
	\noindent As in \cite{A17}, we reduce \Cref{eq:newf,eq:newa,eq:newb} to a system of first-order ODEs. As $f$ does not appear in the equations, the system can be considered first-order in $\dot{f}$. Setting $F = \dot{f}$ and $B=\dot{b}$, we can rewrite the equations as a first-order system in $(F,h,b,B)$ as: 
	
	\begin{align*}
		r\dot{F} &= 
		\frac{1 + b^2 - 4 B^2 h r - b^2 F (1 + 2 F h) r + 4 b B h (-1 + 2 F r)}{2 b^2 h}
		\\
		r\dot{h} &=  - 4hr \frac{B}{b} + 2hrF+r \\
		r\dot{b}&= Br \\
		r\dot{B} &= \frac{1}{4hb}  - B + \frac{rB^2}{b} -\frac{rB}{2h} + \frac{b}{4h}
	\end{align*}
	
	\noindent Defining $u(r,c) \equiv (u_1(r,c), u_2(r,c), u_3(r,c), u_4(r,c)) = (F(r)-c,h(r)-h(0),b(r)-b(0),B(r)-B(0))$, we have an ODE system of the following form with $c$ as a real parameter:
	
	\begin{align}
		\label{scheme}
		r \frac{d u_i}{d r} &= P_i(u,r,c) \\ \nonumber
		u_i(0,c) & = 0 \quad \text{for }i = 1, 2, 3, 4,
	\end{align}
	where $P$ is an analytic function in the neighborhood of the point $(\vec{0},0, c)$ in $\mathbb{C}^6$ and $P(\vec{0},0,c) = 0$. We compute $\frac{\partial P_i} {\partial u_j}$ at $(\vec{0},0,c)$ and obtain
	\begin{equation*}
		\begin{bmatrix}
			0 & -\frac{1+b_0^2}{2b_0^2} & \frac{-1+b_0^2}{2b^3_0}  & -\frac{2}{b_0} \\
			0 & 0&0&0 \\
			0 & 0&0&0 \\
			0 & -\frac{1}{4}(b_0+\frac{1}{b_0}) & \frac{1}{4}(1-\frac{1}{b_0^2}) & -1 \\
		\end{bmatrix}.
	\end{equation*}
	To apply the theorem, we calculate
	\begin{equation*}
		\text{det}\left(m I - \frac{\partial P}{\partial u}\right) = m^3(m+1),
	\end{equation*}
	and this matrix has no positive integer roots, so its inverse exists for all $m \in \mathbb{N}$. We can check that there exists $B \in \mathbb{R}$ such that
	\begin{equation*}
		{\left(m I - \frac{\partial P}{\partial u}\right)^{-1}} < B
	\end{equation*}
	for all $m\in \mathbb{N}$. Thus, by \Cref{thm:existence}, there exists a local solution to the system with the given boundary conditions. This proves the local existence in the case of boundary conditions corresponding to $S^2 \times \mathbb{R}^2$.
	
	\bigskip
	
	\noindent For the $S^1 \times \mathbb{R}^3$ case, we cannot use a similar procedure, as considering $b$ as the independent variable of the soliton equations leads to a system of equations in which the differential equation for the corresponding quantity $r\dot{F}$ retains a singularity at $r=0$. Instead, note that any $\text{SO}(3) \times \text{SO}(2)$-symmetric solitons on $S^1 \times \mathbb{R}^3$ are cohomogeneity one with a singular orbit at $r=0$, falling into the framework of \cite{NW24} in which expanding Ricci solitons of warped product type generalizing the case of $S^1 \times \mathbb{R}^3$ topology are constructed, with local existence of following from results in \cite{Buz11}. Although the soliton equations in \cite{Buz11} are written with respect to different quantities, the results imply the required local existence and uniqueness to \cref{eq:feq,eq:aeq,eq:beq} with initial conditions \cref{S1R3a,S1R3b}, as is also mentioned in Section 1 of \cite{NW24}.

	\setcounter{section}{2}
	\section*{Appendix C: Smooth Cheeger-Gromov Convergence}
	\refstepcounter{section}
	
	\noindent In this section, we prove certain facts about Cheeger-Gromov convergence of cohomogeneity one solitons that have been used extensively in this paper. The results have parallels to Lemma 3.5 of \cite{BHZ22}; in our setup, we have $2$ warping functions instead of one, which introduces slight differences.
	
	\begin{lem}\label{lem:CG}
		Suppose $(M,g, p_i)$ is a sequence of pointed Riemannian manifolds with topology either $S^1 \times \mathbb{R}^3$ or $S^2 \times \mathbb{R}^2$, where the metrics are of the warped product form 
		\begin{center}
			$g_i = dr^2 + a_i(r)^2g_{S^1} + b_i(r)^2g_{S^2}$ 
		\end{center}
		
		\noindent with $r(p_i) = 0$ for each $i$ where the functions $a_i$ and $b_i$ have domain $[-L_i, \infty)$, where $L_i \to \infty$. Suppose the monotonicity bounds $a'_i, b'_i \geqslant 0$ hold as well. Consider bounds of the following form for all $i$:
		
		\begin{center}

			$|\nabla^k \textnormal{Rm}_{g_i}| \leqslant C_k(D)$ on the interval $[-D,D]$ for all $i$, for any $D >0$, $k \geqslant 0$
			
			\medskip
			
			$\alpha_1 \leqslant a_i(0) \leqslant \alpha_2 \hskip 2cm \beta_1 \leqslant b_i(0) \leqslant \beta_2$
		\end{center}
		
		\begin{enumerate}
			\item Suppose we have the above bounds on $g_i$. Then, $(M,g_i,p_i)$ converges in the Cheeger-Gromov sense to a smooth Riemannian manifold $(M_\infty, g_\infty, p_\infty)$ with topology $\mathbb{R} \times S^2 \times S^1$ and a cohomogeneity one metric $g_\infty$.
			\item Suppose we have the curvature bounds, but that either $a_i(0)$ or $b_i(0)$ or both approach infinity. Then $(M,g_i,p_i)$ converges in the Cheeger-Gromov sense to a smooth Riemannian manifold $(M_\infty, g_\infty, p_\infty)$ with topology $\mathbb{R} \times S^2 \times \mathbb{R}^1$ or $\mathbb{R} \times \mathbb{R}^2 \times S^1$ or $\mathbb{R} \times \mathbb{R}^2 \times \mathbb{R}^1$ respectively, and a warped product metric $g_\infty$ which is invariant under the appropriate isometry group (depending on the topology).
		\end{enumerate}
	\end{lem}
	
	\begin{proof}
		1. The sectional curvature bounds give us the following inequalities for all $i$:
		
		\begin{equation}\label{eq:curvature_bounds}
			\displaystyle \left|\frac{a_i''}{a_i}\right| \leqslant C(D) \hskip 1cm \left|\frac{b_i''}{b_i}\right| \leqslant C(D) \hskip 1cm \left|\frac{a_i'b_i'}{a_ib_i}\right| \leqslant C(D) \hskip 1 cm \left|\frac{1-(b'_i)^2}{b^2_i}\right| \leqslant C(D)
		\end{equation}
		
		\noindent \textbf{Step 1: } We derive bounds on $b$. This step is similar to Step 1 in Lemma 3.5 of \cite{BHZ22}.
		
		\medskip
		
		\noindent Consider the fourth equation of \Cref{eq:curvature_bounds}. On the interval $[-(D+1), D+1]$, we have that $(b_i')^2 \geqslant 1 - b_i^2C(D+1)$. From this, we claim that $b_i(-D) \geqslant \text{min}\{ \frac{1}{2C(D+1)^{1/2}}, \frac{1}{2}\}$. If this were false, then we would have the inequality $b_i'(-D) \geqslant \frac{\sqrt{3}}{2}$. Then, we know that for $r \in -[(D+1), -D]$,
		
		\begin{center}
			$b_i'(r)^2 \geqslant 1 - C(D+1)b_i(r)^2 \geqslant 1-C(D+1)b_i(-D)^2 \geqslant \frac{3}{4}$
		\end{center}
		
		\noindent where we used the hypothesis of the monotonicity of $b_i$. Thus, $b_i'(r) \geqslant \frac{\sqrt{3}}{2}$ on $[-(D+1),-D]$. But then, $b_i(-(D+1))$ would become negative since $b_i(-D)$ is also bounded from above by $\frac{1}{2}$. Thus, we have a contradiction, so we must have the bound $b_i(-D) \geqslant \text{min}\{ \frac{1}{2C(D+1)^{1/2}}, \frac{1}{2}\}$. By the monotonicity of $b_i$, this bound holds on $[-D,D]$. Thus, we have
		
		\begin{equation*}
			C(D) \leqslant b_i(r)
		\end{equation*}
		
		\noindent on $[-D,D]$ for all $i$.
		
		\medskip
		
		\noindent Then, again by the fourth equation of \Cref{eq:curvature_bounds}, we know that on $[-D,D]$
		
		\begin{equation*}
			\displaystyle 0 \leqslant \frac{(b_i')^2}{b_i^2} \leqslant C(D) + \frac{1}{b_i^2} \leqslant C(D)+C(D) \equiv C(D)
		\end{equation*}
		
		\medskip 
		
		\noindent by using the lower bound on $b_i$ in the above inequality. Thus, we have the following
		
		\begin{equation}\label{eq:bpb}
			\displaystyle 0 \leqslant \frac{d}{dr} \log(b_i) \leqslant C(D) \hskip 2cm \displaystyle 0 \leqslant \frac{b'_i(0)}{b_i(0)} \leqslant C
		\end{equation}

		\medskip
		
		\noindent By integrating the inequality above, we get a lower bound $c(D, \beta_1) \leqslant b_i(-D)$, which extends to a lower bound on $[-D,D]$ by monotonicity.

		\medskip
		
		\noindent By \Cref{eq:bpb} and the hypothesis that $b_i(0) \leqslant \beta_2$, we have a bound of the form $b_i'(0) \leqslant C(\beta_2)$. Then, using the second equation of \Cref{eq:curvature_bounds}, we know that since $b_i(0)$ and $b'_i(0)$ are bounded from above, we can integrate the bound $b_i'' \leqslant C(D)b_i$ to get an exponential growth upper bound for $b_i(D)$. By the monotonicity of $b_i$, this bound holds on $[-D,D]$. Thus, to summarize, we have the following bounds on $[-D,D]$ for $b_i$ for all $i$:
		
		\begin{equation}\label{eq:bb}
			c(D, \beta_1) \leqslant b_i(r) \leqslant C(D, \beta_2)
		\end{equation}
		
		\medskip
		
		\noindent \textbf{Step 2: } By the first equation of \Cref{eq:curvature_bounds}, we know that $|a_i''| \leqslant C(1)a_i$ on $[-1,1]$. Applying this to the subinterval $[-1,0]$, we know that by the monotonicity of $a_i$ that $a_i \leqslant \alpha_2$ on $[-1,0]$. Thus, we have $|a_i''| \leqslant C(1)\alpha_2$ on $[-1,0]$. This implies that $a'_i(0)$ is bounded from above by a constant $C(\alpha_2)$; otherwise, $a'_i$ would be very large on the interval $[-1,0]$ and $a_i(-1)$ would be negative. Then, as in Step 1, we can get an exponential growth upper bound for $a_i(D)$ of the form $a_i(D) \leqslant C(D, \alpha_2)$ by integrating the inequality $|a''_i| \leqslant C(D)a_i$.
		
		\medskip
		
		\noindent By monotonicity and the lower bounds $a_i(0) \geqslant \alpha_1$, and $b_i(0) \geqslant \beta_1$, we have lower volume bounds on small $s$-balls for any point $q$ with $r=0$. By the curvature bound, by the Bishop-Gromov inequality, we have lower volume bounds of $s$-balls (whose centers are at distance at most $D$ to $q$) by constants $C(D, \alpha_1, \beta_1)$. From this, we have a lower bound of the form $a(-D) \geqslant C(D, \alpha_1, \beta_1)$. By monotonicity, this bound holds on $[-D,D]$. Thus, to summarize, we have the following bounds on $[-D,D]$ for $a$:
		
		\begin{equation}\label{eq:ab}
			c(D, \alpha_1, \beta_1) \leqslant a_i(r) \leqslant C(D, \alpha_2)
		\end{equation}
		
		\medskip
		
		\noindent \textbf{Step 3: } By \Cref{eq:bb} and \Cref{eq:ab}, we have upper bounds on $a_i$ and $b_i$ by constants of the form $a_i \leqslant C(D, \alpha_2)$ and $b_i \leqslant C(D, \beta_2)$, respectively. Using the first and second equations of \Cref{eq:curvature_bounds}, on $[-D,D]$ we have bounds of the form (for all $i$)
		
		\begin{center}
			$|a_i'| \leqslant C(D, \alpha_2)$ \hskip 1cm $|b_i'| \leqslant C(D, \beta_2)$
		\end{center}
		
		\noindent In addition, using the lower bounds on $a$ and $b$ on $[-D,D]$, we have for all $i$
		
		\begin{equation}
			\displaystyle \left|\frac{a_i'}{a_i}\right| \leqslant C(D, \alpha_1, \alpha_2, \beta_1) \hskip 1cm \left|\frac{b_i'}{b_i}\right|\leqslant C(D, \beta_1, \beta_2)
		\end{equation}

		\noindent \textbf{Step 4: } Now, the curvature derivative bounds imply the following inequalities for all $i$:
		
		\begin{equation}\label{eq:curvature_derivative_bounds}
			\displaystyle \left|\frac{d^k}{dr^k} \left(\frac{a_i''}{a_i} \right)\right| \leqslant C_k(D) \hskip 2cm \left|\frac{d^k}{dr^k} \left( \frac{b_i''}{b_i}\right)\right| \leqslant C_k(D)
		\end{equation}
		
		\noindent We use these bounds to prove bounds of the form on $[-D,D]$ for all $i$
		
		\begin{equation}\label{eq:aballbounds}
			\displaystyle |a_i^{(k)}| \leqslant C(C_0, \cdots, C_k, D, \alpha_1, \alpha_2, \beta_1) \hskip 2cm|b_i^{(k)}| \leqslant C(C_0, \cdots, C_k, D, \beta_1, \beta_2)
		\end{equation}
		
		\noindent The $k=0$ and $k=1$ cases are taken care of by Steps 1 to 3. For $k \geqslant 2$, we use \Cref{eq:curvature_derivative_bounds} and the bounds on $a'_i/a_i$ and $b'_i/b_i$ along with the bounds on $a_i$ and $b_i$ derived above to prove \Cref{eq:aballbounds} by induction. Then, by Arzela-Ascoli, we have subsequential convergence in $C^{\infty}_{loc}(\mathbb{R})$ of $a_i$ and $b_i$ to smooth positive functions $a_\infty, b_\infty: \mathbb{R} \to \mathbb{R}^+$. Thus, we have a smooth metric $g_\infty = dr^2 + a_\infty(r)^2g_{S^1}+b_\infty(r)^2g_{S^2}$ on $\mathbb{R} \times S^2 \times S^1$.
		
		\medskip
		
		\noindent By hypothesis, our manifolds have topology $S^1 \times \mathbb{R}^3$ or $S^2 \times \mathbb{R}^2$, with a singular orbit at $r = -L_i$. Now, with $U_i:=(-L_i/2, L_i/2) \times S^2 \times S^1$, consider the maps $\phi_i : U_i \to M$, where $\phi_i$ maps the points with coordinates $(r,z_2, z_1)$ to the point in $M$ in the orbit at distance $r+L_i$ from the singular orbit and whose coordinates on $S^2$ and $S^1$ are $z_2$ and $z_1$, respectively. Note that we have chosen $p_i \in (M,g_i)$ to have $r(p_i)=0$ for all $i$. Then, we have $\phi^*_ig_i-g_\infty \to 0$ in $C^{\infty}_{loc}(\mathbb{R}\times S^2 \times S^1)$. Passing to a further subsequence, we get the convergence $\phi^{-1}_i(p_i) \to p_\infty$.

		\bigskip
		
		\noindent For 2, for ease of notation we assume that both orbit sizes at $0$ blow up; the proof is similar if only one of them does. We consider the functions $\tilde{a}_i := \frac{a_i}{a_i(0)}$, $\tilde{b}_i := \frac{b_i}{b_i(0)}$. Note that the same bounds as in \Cref{eq:bpb} to \Cref{eq:aballbounds} can be proven for these functions, except without any reference to the $\alpha_i$ and $\beta_i$ constants. Thus, by Arzela-Ascoli, we have convergence in $C^{\infty}_{loc}(\mathbb{R})$ of $\tilde{a}_i$ and $\tilde{b}_i$ to smooth positive functions $a_\infty, b_\infty: \mathbb{R} \to \mathbb{R}^+$.
		
		\medskip
		
		\noindent Now, consider the smooth metric $g_\infty = dr^2 + a_\infty(r)^2g_{\mathbb{R}^1}+b_\infty(r)^2g_{\mathbb{R}^2}$ on $\mathbb{R} \times \mathbb{R}^2 \times \mathbb{R}^1$. Since $a_i(0)$ and $b_i(0)$ converge to $\infty$, it is easy to see that there are diffeomorphisms $\phi_{1,i} : (U_{1,i}, g_{E_1}) \to (S^1, g_{a_i(0)})$ and $\phi_{2,i} : (U_{2,i}, g_{E_2}) \to (S^2, g_{b_i(0)})$ (where $(U_{k,i}, g_{E_k})$ is a subset of $\mathbb{R}^k$ with the Euclidean metric and $g_{a_i(0)}$ and $g_{b_i(0)}$ are respectively the metrics on $S^2$ and $S^1$ of the sizes in the subscripts and the $U_{k,i}$ cover $\mathbb{R}^k$) such that $\phi_{1,i}^*g_{a_i(0)} \to g_{E_1}$ in $C^\infty_{loc}(\mathbb{R})$ and $\phi_{2,i}^*g_{b_i(0)} \to g_{E_2}$ in $C^\infty_{loc}(\mathbb{R}^2)$. Now, consider the diffeomorphisms $\phi_i: (-L_i/2, L_i/2) \times U_{1,i} \times U_{2,i} \to (M,g_i)$ which map $(r,z_1,z_2)$ to the point in $M$ at distance $r+L_i$ from the singular orbit and whose coordinates on $S^2$ and $S^1$ are respectively given by $\phi_{2,i}(z_2)$ and $\phi_{1,i}(z_1)$. Then, we have $\phi^*_ig_i-g_\infty \to 0$ in $C^{\infty}_{loc}(\mathbb{R}\times \mathbb{R}^2 \times \mathbb{R}^1)$. Convergence of $\phi^{-1}_i(p_i)$ follows from the fact that $r(p_i)=0$ and by the fact that the convergence of the spheres to Euclidean spaces can always be chosen to be pointed (by adjusting by translations if necessary).
	\end{proof}

\end{document}